\documentclass{aic}

\usepackage{subcaption}
\usepackage{float}
\newtheorem{theorem}{Theorem}[section]
\newtheorem{lemma}[theorem]{Lemma}
\newtheorem{definition}[theorem]{Definition}
\newtheorem*{question}{Question}
\newtheorem{corollary}[theorem]{Corollary}

\newcommand{\B}{\mathcal{B}}
\newcommand{\PP}{\mathcal{P}}
\newcommand{\C}{\mathcal{C}}
\newcommand{\Sp}{\mbox{Sp}}

\aicAUTHORdetails{%
  title = {Cycle Lengths Modulo $k$ in Large 3-connected Cubic Graphs}, 
  author = {Kasper Szabo Lyngsie and Martin Merker},
  plaintextauthor = {Kasper Szabo Lyngsie, Martin Merker},
  plaintexttitle = {Cycle Lengths Modulo k in Large 3-connected Cubic Graphs}, 
}   

\aicEDITORdetails{%
   year={2021},
   number={3},
   received={11 April 2019},   
   revised={5 June 2020},    
   published={3 February 2021},  
   doi={10.19086/aic.18971},      
}   

\begin{document}

\begin{frontmatter}[classification=text]


\author[kasper]{Kasper Szabo Lyngsie}
\author[martin]{Martin Merker\thanks{Supported by the Danish Council for Independent Research, Natural Sciences, grant DFF-8021-00249, AlgoGraph.}}

\begin{abstract}
We prove that for all natural numbers $m$ and $k$ where $k$ is odd, there exists a natural number $N(k)$ such that any 3-connected cubic graph with at least $N(k)$ vertices contains a cycle of length $m$ modulo $k$. We also construct a family of graphs showing that this is not true for 2-connected cubic graphs if $m$ and $k$ are divisible by 3 and $k\geq 12$.
\end{abstract}
\end{frontmatter}

\section{Introduction}

\subsection{Cycle lengths modulo $k$}

Let $G$ be a graph and let $A$ be a set of natural numbers. Which properties of $G$ and $A$ guarantee the existence of a cycle in $G$ whose length is in $A$? Erd\H{o}s asked many years ago whether there exists a set $A \subset \mathbb{N}$ of density zero and two constants $c_A,n_A$ such that every graph with at least $n_A$ vertices and average degree at least $c_A$ contains a cycle whose length is in $A$, see~\cite{chung}. Verstraëte~\cite{verte} answered this question in the affirmative by showing that any graph with average degree at least 10 contains a cycle whose length is in a prescribed set $A \subset \mathbb{N}$ satisfying $|A \cap \{1,\ldots,n\}|=O(n^{0.99})$. 

Of particular interest in the literature is the case where $A$ is a residue class modulo some natural number $k$. There are many results concerning sufficient conditions for the existence of a cycle of length $m$ modulo $k$ where $m\in\{0,1,2\}$ and $k\in \{3,4\}$, see for example~\cite{deanlesniak, chensaito, mei, saito}.
In 1976, Erd\H{o}s and Burr~\cite{erdos} conjectured that for all natural numbers~$m$ and $k$ where $k$ is odd, there exists a constant $c_k(m)$ such that every graph with average degree at least $c_k(m)$ has a cycle of length $m$ modulo $k$. Note that the restriction to odd natural numbers $k$ is necessary since bipartite graphs contain no cycles of odd length.
Bollobás~\cite{bollobas} showed that $c_k(m)=\frac{2}{k}((k+1)^k-1)$ suffices, proving the conjecture by Erd\H{o}s and Burr. Sudakov and Verstraëte~\cite{sudakov} showed that $c_k(m) = O(mk^{\frac{2}{m}})$.

In 1983, Thomassen~\cite{thomassenpath} conjectured that for all natural numbers $m$ and $k$, every graph of minimum degree at least $k+1$ contains a cycle of length $2m$ modulo $k$. Thomassen showed that minimum degree $4m(k+1)$ suffices. Cai and Shreve~\cite{caishreve} showed that claw-free graphs of minimum degree $k+1$ have cycles of all lengths modulo $k$. Diwan~\cite{diwan} proved that graphs of minimum degree $2k-1$ have cycles of all even lengths modulo $k$ and that Thomassen's conjecture holds for $m=2$. The currently best known result is by Liu and Ma~\cite{liuma} who verified Thomassen's conjecture if $k$ is even and showed that minimum degree $k+4$ suffices if $k$ is odd. 

A commonly used method to show that a graph has cycles of every even length modulo~$k$ is to construct a sequence of cycles whose lengths form an arithmetic progression with difference 2. Bondy and Vince~\cite{bondy} answered a question of Erd\H{o}s by showing that a graph with minimum degree 3 contains two cycles whose lengths differ by 1 or 2. Fan~\cite{fan} showed that graphs of minimum degree $3k-2$ contain $k$ cycles whose lengths form an arithmetic progression with difference 2. A similar result was proved by Verstraëte~\cite{verstarit} who showed that graphs of average degree $8k$ and even girth $g$ contain $(\frac{g}{2}-1)k$ cycles of consecutive even lengths.
Sudakov and Verstraëte~\cite{sudakov2} showed that a graph with average degree $d$ and girth~$g$ has $\Omega(d^{\lfloor \frac{g-1}{2} \rfloor})$ cycles of consecutive even lengths, proving a conjecture by Erd\H{o}s about the minimum number of distinct cycle lengths. Ma~\cite{ma} proved an analogous result about cycles with consecutive odd lengths in non-bipartite 2-connected graphs. Liu and Ma~\cite{liuma} showed that every graph with minimum degree $k+1$ contains $\lfloor \frac{k}{2}\rfloor$ cycles of consecutive even lengths.

Most sufficient conditions for the existence of cycles of length $m$ modulo $k$ require a minimum degree which grows linearly in $k$. The reason for that is that graphs where every block is a clique on at most $k+1$ vertices contain no cycles of length 2 modulo $k$. In this paper we focus on cycles in cubic graphs. Thomassen~\cite{girth} proved that a graph with minimum degree at least 3 and girth at least $2(k^2+1)(3 \cdot2^{k^2+1}+(k^2+1)^2-1)$ contains cycles of all even lengths modulo $k$, which is the strongest known result for the class of cubic graphs. We prove a similar statement for cubic graphs under the mild assumption that the graph is sufficiently large and 3-connected.

\begin{theorem}\label{thm:main}
For every odd natural number $k$, there exists a natural number $N(k)$ such that every 3-connected cubic graph with at least $N(k)$ vertices contains a cycle of length $m$ modulo $k$ for every natural number $m$.
\end{theorem}

Theorem~\ref{thm:main} does not hold if $k$ is even since bipartite graphs have no cycles of odd length. If $k$ is even and not divisible by 4, then our proofs imply that the set of cycle lengths of every sufficiently large cubic 3-connected graph contains all odd residues or all even residues modulo $k$. It looks plausible that it will always contain all even residues, but our methods are not sufficient to prove it. Similarly, if $k$ is divisible by 4, then we can show that the set of cycle lengths will contain at least a quarter of all residues modulo $k$. To be more precise, for some $i\in \{0,1,2,3\}$, it will contain all residues that are congruent to $i$ modulo 4.\\ 
Theorem~\ref{thm:main} is not true for 3-connected graphs of minimum degree $d$ where $d\geq 3$ and $d<\frac{k}{2}$. To see this, note that for $n\geq d$ the complete bipartite graph $K_{d,n}$ is 3-connected and contains no cycle of length divisible by $k$. However, it might be true for $d$-regular graphs. The methods in this paper have been tailored to cubic graphs, but the general ideas extend to regular graphs. In particular, we believe that the structures we find in large cubic graphs also exist in large $d$-regular graphs. An extension of the proof to $d$-regular graphs might be possible but there appear to be various technical obstacles.\\
We also show that the connectivity condition in Theorem~\ref{thm:main} cannot be lowered to 2-connectivity in general. Given natural numbers $m$, $k$, and $N$ such that $k \geq 12$ and $m$ and $k$ are divisible by 3, we construct a 2-connected cubic graph on at least $N$ vertices which has no cycles of length $m$ modulo~$k$.

\subsection{Proof overview and structure of the paper}

In the proof of Theorem~\ref{thm:main}, we show that in large 3-connected cubic graphs there exists a subgraph of a certain structure which has cycles of every length modulo $k$. The structure we use is necklace-like: it consists of a collection of 2-connected subgraphs which are joined by paths in a cyclic order. The general idea is that in such a structure we can take a cycle and modify it locally inside one of the 2-connected subgraphs to obtain a new cycle which ideally has a different length modulo $k$. If there are sufficiently many 2-connected subgraphs, then we have enough freedom to find all cycle lengths modulo $k$. Of course we need additional assumptions to ensure that there is a local modification which changes the length modulo $k$. We distinguish two types of such necklaces which are defined in Section~2.

The first type of necklace is a so-called $\theta$-necklace. In this case there are several 2-connected subgraphs, each consisting of a cycle and a path of a given length which attaches to the cycle. In this case we cannot bound the difference of the cycle lengths if we change the cycle locally. However, if the length of the path is a power of 2, then we can obtain all residues modulo $k$ for odd $k$ by combining local modifications. We usually show the existence of a $\theta$-necklace by finding a path and a cycle which are joined by many disjoint paths of the same length. We call such a cycle and path \emph{close} since we are mainly interested in the case where the paths have length at most 2. To find a cycle and a path which are close, we can pick a cycle $C$ in $G$ and then try to find a path which contains many neighbours of $C$, or vertices at distance 2. In Section 3 we show that we can find such a path provided $G-C$ has a block containing sufficiently many neighbours of $C$.

In the second type of necklace we require many 2-connected parts where we can change the length of the cycle by 1 or 2. We call such a necklace \emph{wiggly}. In wiggly necklaces we can find a sequence of cycles such that their lengths form an arithmetic progression with difference 1 or 2. For our purposes it is only important that the differences are a power of 2 to make sure that we can obtain all residues modulo $k$ for odd $k$. To show the existence of wiggly necklaces, we need two paths between two given vertices $x$ and $y$ whose lengths differ by 1 or 2. A result by Fan~\cite{fan} implies that such paths exist in a subcubic 2-connected graph if $x$ and $y$ are the only vertices of degree~2. However, in the main proof we need to find these paths when there is also a third vertex of degree 2 (and in some cases even a fourth such vertex). Thus, we need to prove extensions of Fan's results in the case of subcubic graphs, which we do in Section~4.

Section~5 contains the main part of the proof of Theorem~\ref{thm:main}. We start by choosing a subgraph $\Theta$ in $G$ which consists of three internally disjoint paths between two vertices. There are several cases depending on the number of 2-connected endblocks in $G-\Theta$ and how these endblocks attach to $\Theta$. If there are many 2-connected endblocks which have many neighbours on $\Theta$, then we find a $\theta$-necklace. If there are many 2-connected endblocks with only few neighbours on $\Theta$, then we find a wiggly necklace. Finally, if there are only few 2-connected endblocks, then we also find a $\theta$-necklace. For technical reasons the case distinction in Section 5 is slightly different. We distinguish between $\Theta$-isolated endblocks which are endblocks with neighbours on only one of the three paths in $\Theta$ and $\Theta$-connecting endblocks which have neighbours on at least two different paths. 

In Section 6, we show that Theorem~\ref{thm:main} does not hold in general for 2-connected cubic graphs. For $k\geq 12$ and $k$ divisible by 3, we construct an infinite family of 2-connected graphs which do not contain cycles of every length modulo $k$.

\subsection{Notation and preliminaries}

All graphs in this paper are simple and finite unless stated otherwise. We denote the vertex set and edge set of a graph $G$ by $V(G)$ and $E(G)$, respectively. If $H$ is a subgraph of $G$ or $H\subseteq V(G)$, we write $G-H$ or $G-V(H)$ to denote the graph obtained from $G$ by removing all vertices in $H$. We write $G[A]$ to denote the subgraph of $G$ induced by the vertices in $A$. If $v\in V(G)$ and $e\in E(G)$, then $G-v$ and $G-e$ are the graphs obtained from $G$ by removing $v$ and $e$, respectively. We write $d(v)$ and $N(v)$ for the degree and the neighbourhood of $v$, respectively. If $H$ is a subgraph of $G$, then $N_{H}(u)= N(u) \cap V(H)$ and $d_H(u)=|N_H(u)|$. If $A\subseteq V(G)$ or $A$ is a subgraph of a graph $G$, then  $N(A)$ denotes the set of vertices in $G - A$ which have a neighbour in $A$. If $B\subseteq V(G)$ or $B$ is a subgraph of~$G$, then $N_B(A)$ denotes the set of vertices in $B$ which have a neighbour in $A$ and $E(A,B)$ denotes the set of edges having an end in both $A$ and $B$. 
If $H$ is a subgraph of a graph $G$ and $x,y \in V(H)$, then $H+xy$ is the graph obtained from $H$ by adding the edge $xy$.  
A path from a vertex $u$ to a vertex $v$ is called a $u-v$ path and, more generally, if $A$ and $B$ are vertex sets or subgraphs of a graph $G$, then an $A - B$ path is a path having one end in $A$, one end in $B$, and no internal vertices in $A\cup B$. If $x$ and $y$ are vertices on a path~$P$, then $xPy$ denotes the $x-y$ subpath of $P$. If $G$ is connected and $u,v\in V(G)$, then we write $dist(u,v)$ for the length of a shortest $u-v$ path in $G$. 

A \emph{block} in a graph $G$ is a maximal connected subgraph of $G$ without any cut-vertices. Note that a block of a graph $G$ is either a maximal 2-connected subgraph of $G$, a bridge in~$G$ or an isolated vertex. The \emph{block graph} $B(G)$ of a graph $G$ is the bipartite graph whose vertex set consists of the cut-vertices and the blocks in $G$, and whose edges are of the form $vB$ where $v$ is a cut-vertex in $G$ and $B$ is a block containing $v$. If $G$ is connected, then $B(G)$ is a tree. A block in $G$ corresponding to a vertex of degree at most 1 in $B(G)$ is called an \emph{endblock}. We call a 2-edge-cut \emph{non-trivial} if it separates $G$ into two components each containing at least two vertices.

The following lemma is well-known.

\begin{lemma}\label{lem:path_system}
Let $G$ be a connected graph and let $S \subset V(G)$. If $|S|$ is even, then there exist $\frac{|S|}{2}$ pairwise edge-disjoint paths with endvertices in $S$ such that each vertex of $S$ is an endvertex of exactly one of the paths.
\end{lemma}
\begin{proof}
We may assume that $G$ is a tree. Let $\PP$ be a collection of $\frac{|S|}{2}$ paths such that each vertex of $S$ is an endvertex of precisely one of them, and such that sum of the lengths of the paths in $\PP$ is minimal. It follows easily that the paths in $\PP$ are edge-disjoint.
\end{proof}

We will apply Lemma~\ref{lem:path_system} where $G$ is a subcubic graph. Note that in this case any two edge-disjoint paths are also internally vertex-disjoint.

The following classical result by Erd\H{o}s and Szekeres is used several times throughout the paper.

\begin{theorem}[Erd\H{o}s, Szekeres \cite{erdossz}]\label{thm:erdos-szekeres}
Every sequence of at least $(r-1)(s-1)+1$ distinct elements of an ordered set contain an increasing subsequence of length $r$ or a decreasing subsequence of length $s$. 
\end{theorem}

A cycle $C$ in a graph $G$ is called \emph{non-separating} if $G-V(C)$ is connected. In Section~4 the existence of non-separating cycles plays an important role. We use the following theorem by Tutte on non-separating cycles in 3-connected graphs. 

\begin{theorem}[Tutte \cite{tutte}] \label{thm:tutte}
Let $G$ be a graph, $st \in E(G)$, and  $r \in V(G)\setminus\{s,t\}$. If $G$ is 3-connected, then $G$ contains a non-separating induced cycle $C$ such that  $st\in E(C)$ and $r\notin V(C)$.
\end{theorem}

\section{Necklaces}

In this section we define two types of necklaces: $\theta$-necklaces and wiggly necklaces. We prove sufficient conditions for the existence of cycles of every length modulo $k$ in these necklaces. We begin by defining $\theta$-graphs, which are the building blocks of $\theta$-necklaces.

\begin{definition}[$\theta$-graph]
A \textbf{$\theta$-graph} $\Theta$ consists of two vertices $u$ and $v$ and three $u-v$ paths $P_1$, $P_2$, and $P_3$ which are pairwise internally disjoint. We call $P_1$, $P_2$, and $P_3$ the \textbf{legs} of $\Theta$ and write $\Theta =(P_1,P_2,P_3)$. We also call $\Theta$ a \textbf{$u,v$-$\theta$-graph}.
\end{definition}

Note that if $x$ and $y$ are vertices on different legs of a $\theta$-graph $\Theta$, then $\Theta$ contains four $x-y$ paths. Under the mild assumption that the length of the third path has no common divisors with $k$, this guarantees the existence of $x-y$ paths with different lengths modulo~$k$. A $\theta (k,i)$-necklace consists of $k$ such $\theta$-graphs where the third path has length $i$ and which are connected by paths in a cyclic order.
 
\begin{definition}[$\theta(k,i)$-necklace]\label{def:theta-necklace}
Let $k$ and $i$ be natural numbers. A $\theta(k,i)$-necklace is a connected subcubic graph consisting of $k$ disjoint $\theta$-graphs $\Theta_1$,\ldots ,$\Theta_k$ and $k$ disjoint paths $P_1$, \ldots ,$P_k$ such that
\begin{itemize}
    \item $\Theta_j$ has legs $L_{j,1}$, $L_{j,2}$, $L_{j,3}$ where $L_{j,3}$ is a path of length $i$ for $j\in\{1,\ldots ,k\}$,
    \item $P_j$ is an $L_{j,2}-L_{j+1,1}$ path for $j\in\{1,\ldots ,k-1\}$, $P_{k}$ is an $L_{k,2}-L_{1,1}$ path, and
    \item for every $j,j'\in\{1,\ldots ,k\}$, no interior vertex of $P_j$ is contained in $\Theta_{j'}$.
\end{itemize}
\end{definition}

See Figure~\ref{fig:two_necklaces}(a) for an example of a $\theta(4,2)$-necklace. A graph is a $\theta$-necklace if it is a $\theta (k,i)$-necklace for some natural numbers $k$ and $i$. The following lemma gives a sufficient condition for a $\theta$-necklace to contain cycles of every length modulo $k$. In the main proof we will only construct $\theta (k,i)$-necklaces where $i\in \{1,2\}$.

\begin{lemma}\label{lem:theta_neck_pan_cyc}
Let $m$, $k$, and $i$ be natural numbers and let $G$ be a $\theta(3k^4,i)$-necklace. If $\gcd (k,2i)=1$, then $G$ has a cycle of length $m$ modulo $k$.
\end{lemma}
\begin{proof}
Let $\Theta_j$, $L_{j,1}$, $L_{j,2}$, $L_{j,3}$, and $P_j$ (for $j\in \{1,\ldots ,3k^4\}$) be as in Definition~\ref{def:theta-necklace}. For $j\in$~$\{1,\ldots ,3k^4\}$, let $u_j$, $v_j$ denote the vertices in $L_{j,1}\cap L_{j,2}\cap L_{j,3}$ and let $w_{j,1}\in V(L_{j,1})$ and $w_{j,2}\in V(L_{j,2})$ denote the ends of $P_{j-1}$ and $P_{j}$, respectively. Let $a_j,b_j,c_j,d_j$  be integers such that $|E(w_{j,1}L_{j,1}u_j)|=a_j$, $|E(u_jL_{j,2}w_{j,2})|=b_j$, $|E(w_{j,1}L_{j,1}v_j)|=c_j-i$, and $|E(v_jL_{j,2}w_{j,2})|=d_j-i$. Note that $\Theta_j$ contains four $w_{j,1}-w_{j,2}$ paths whose lengths are $a_j+b_j$, $a_j+d_j$, $c_j+b_j$, and $c_j+d_j-2i$.\\ 
Let $C$ be the cycle in $G$ obtained by taking the union of all the paths $P_j$, $w_{j,1}L_{j,1}u_j$, and $u_jL_{j,2}w_{j,2}$ for $j \in \{1,\ldots,3k^4\}$. We can modify $C$ by replacing the path within $\Theta_j$ by a different $w_{j,1}-w_{j,2}$ path, which increases the length of $C$ by $d_j-b_j$, $c_j-a_j$, or $c_j+d_j-2i-a_j-b_j$. 
By the pigeonhole principle there are at least $3k$ indices $j$ for which the 3-tuples $(d_j-b_j, c_j-a_j, c_j+d_j-2i-a_j-b_j)$ are equal modulo $k$. Thus, there exists a subset $I \subset \{1,\ldots,3k^4\}$ with $|I| \geq 3k$ and integers $a,b,c,d$, such that for $j\in I$ we have
\begin{align*}
    d_j-b_j & \equiv d-b & \mod{k}\\
    c_j-a_j & \equiv c-a & \mod{k}\\
    c_j+d_j-2i-a_j-b_j & \equiv c+d-2i-a-b & \mod{k}
\end{align*}
By possibly modifying the cycle $C$ inside each $\Theta_j$ for $j \in I$ we can obtain all cycle lengths of the form $|E(C)|+\lambda_1(d-b)+\lambda_2(c-a)+\lambda_3(c+d-2i-a-b) \mod{k}$ where $\lambda_1$, $\lambda_2$, $\lambda_3$ are non-negative integers with $\lambda_1 +\lambda_2  +\lambda_3 \leq 3k$. For every integer $\lambda$, we can obtain all cycle lengths of the form $|E(C)|+2i\lambda \mod{k}$ by choosing $\lambda_1, \lambda_2, \lambda_3 \in \{0,\ldots,k-1\}$ such that $\lambda_1 \equiv \lambda_2 \equiv \lambda \equiv -\lambda_3 \mod{k}$. Since $\gcd (k,2i)=1$ this implies that $G$ contains cycles of length $m$ modulo $k$ for every integer $m$.
\end{proof}

\begin{figure}[t]
    \centering
      \begin{subfigure}[b]{0.37\textwidth}
        \includegraphics[width=\textwidth]{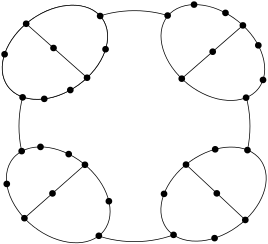}
        \caption{A $\theta(4,2)$-necklace}
        \label{fig:theta2neck}
    \end{subfigure}
    \hspace{2cm}
    \begin{subfigure}[b]{0.41\textwidth}
        \includegraphics[width=\textwidth]{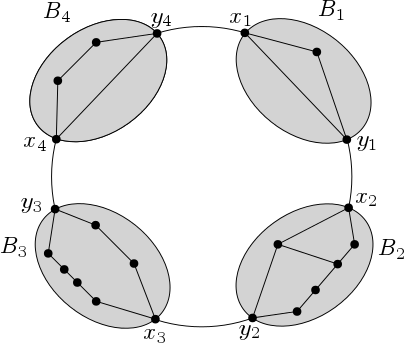}
        \caption{A 4-wiggly necklace}
        \label{fig:wigglyneck}
    \end{subfigure}
    \caption{Two types of necklaces}
    \label{fig:two_necklaces}
\end{figure}

If $k$ is even, then the proof above still shows that for a $\theta(3k^4,i)$-necklace $\Theta$, there is a natural number $c$ such that for every integer $\lambda$ there is a cycle of length $c + 2i\lambda$ modulo $k$ in $\Theta$. However, now this set will not cover all residues. If $i=1$ then the set will contain either all even or all odd residues modulo $k$. This is also the case if $i=2$ and $k$ is not divisible by 4. If $i=2$ and $k$ is divisible by 4, then the set will contain only those residues that are in the same residue class modulo 4 as $c$.\\
As can be seen below, one way to find a $\theta(k,1)$-necklace or a $\theta(k,2)$-necklace inside a cubic graph $G$ is to find a cycle and a path which are disjoint and have many disjoint paths of length at most 2 connecting them. This motivates the following definition.

\begin{definition}[$k$-close]
Two subgraphs $H_1$ and $H_2$ of a graph $G$ are \textbf{$k$-close} if $H_1$ and $H_2$ are disjoint and there exist $k$ pairwise disjoint $H_1-H_2$ paths of length at most 2.
\end{definition}

Sometimes we will slightly abuse this notation by saying that $P_1$ is $k$-close to $P_2\cup P_3$ where $P_1$, $P_2$, and $P_3$ are the legs of a $\theta$-graph. Note that in this case the subpath induced by the interior vertices of $P_1$ is $k$-close to $P_2\cup P_3$.

\begin{lemma}\label{lem:kclose}
If a graph contains a cycle and a path which are $18k^2$-close, then it contains a $\theta(k,1)$-necklace or a $\theta(k,2)$-necklace.
\end{lemma}
\begin{proof}
Let $G$ be a graph containing a path $P=p_1p_2 \ldots p_n$ and a cycle $C=c_1c_2 \ldots c_mc_1$ which are $18k^2$-close. Let $C'=C-c_mc_1$. Either $|E(P,C)|\geq 9k^2$ or $P$ and $C$ are joined by at least $9k^2$ disjoint paths of length 2. Suppose $|E(P,C)|\geq 9k^2$. Define a total ordering of $V(P)$ by $p_i \leq p_j$ if and only if $i \leq j$ and a total ordering of $V(C)$ by $c_i \leq c_j$ if and only if $i \leq j$.
Given an edge $e$ between $P$ and $C$ let $p(e)$ denote the end of $e$ on $P$ and let $c(e)$ denote the end of $e$ on $C$. Let $e_1,\ldots,e_{9k^2}$ be distinct edges between $P$ and $C$ such that $p(e_i) \leq p(e_j)$ whenever $i \leq j$. Let $S=(c(e_1),c(e_2),\ldots,c(e_{9k^2}))$. Note that $9k^2 \geq (3k-1)^2+1$ so by Theorem~\ref{thm:erdos-szekeres} $S$ has a subsequence $S'=(c(e'_1),\ldots,c(e'_{3k}))$ of length $3k$ which is increasing or decreasing.
For each $i \in \{1,\ldots,k\}$, let $C_i$ be the cycle obtained from the union of the paths $p(e'_{3i-2})Pp(e'_{3i})$ and $c(e'_{3i-2})C'c(e'_{3i})$ together with the edges $e'_{3i-2}$ and $e'_{3i}$ Note that each $C_i$ has a chord $e'_{3i-1}$. Now the union of the $\theta$-graphs $C_i+e_{3i-1}'$ and the cycle $C$ forms a $\theta(k,1)$-necklace.\\
If $P$ and $C$ are joined by at least $9k^2$ disjoint paths of length 2 the proof is analogous. The only difference is that the chords of the cycles $C_i$ are now subdivided once. Thus, we obtain a $\theta(k,2)$-necklace in this case. 
\end{proof}

The following definition is similar to that of a $\theta$-necklace but allows more general building blocks.

\begin{definition}[$\mathcal{B}$-necklace, $k$-wiggly necklace]\label{def:wigglynecklace}
Let $k$ and $\ell$ be natural numbers. Let $\mathcal{B} = \{B_1,\ldots ,B_\ell\}$ be a collection of 2-connected pairwise disjoint subgraphs of a graph $G$. A \textbf{$\mathcal{B}$-necklace} $N$ is a subgraph of $G$ consisting of $B_1$,\ldots ,$B_\ell$ and $\ell$ pairwise disjoint paths $P_1$, \ldots ,$P_\ell$ such that 
\begin{itemize}
    \item $P_i$ is a $B_i-B_{i+1}$ path for $i\in\{1,\ldots ,\ell-1\}$, $P_{\ell}$ is a $B_\ell-B_1$ path, and
    \item for every $i,j\in\{1,\ldots ,\ell\}$, no interior vertex of $P_i$ is contained in $B_j$.
\end{itemize}
For $i\in \{1,\ldots ,\ell\}$, let $x_i$, $y_i$ denote the two vertices of $B_i$ which are contained in some path $P_j$. We say $N$ is \textbf{$k$-wiggly} if there exists $I\subseteq \{1,\ldots ,\ell\}$ with $|I|=k$ such that $B_i$ contains two $x_i-y_i$ paths whose lengths differ by 1 or 2 for every $i\in I$.
\end{definition}

Note that the two paths inside each $B_i$ are not necessarily disjoint. For example, $B_2$ in the 4-wiggly necklace in Figure~\ref{fig:two_necklaces}(b) contains $x_2-y_2$ paths of lengths 2, 4, and 5, but no two disjoint $x_2-y_2$ paths whose lengths differ by 1 or 2.\\
In a $k$-wiggly necklace there exists a cycle which we can modify locally in $k$ places so that each time its length increases by 1 or 2. We use this to show that $2k$-wiggly necklaces contain cycles of every length modulo $k$.

\begin{lemma}\label{lem:wiggly_pan_cyc}
Let $k$ and $m$ be natural numbers with $k$ odd. If a necklace is $2k$-wiggly, then it has a cycle of length $m$ modulo $k$.
\end{lemma}
\begin{proof}
Let $G$ be a $2k$-wiggly necklace with 2-connected subgraphs $B_0,\ldots,B_{2k-1}$ and paths $P_0,\ldots,P_{2k-1}$ as in Definition~\ref{def:wigglynecklace}. For each $j\in \{0,\ldots ,2k-1\}$, let $P_{j,1}, P_{j,2}$ be two $P_{j-1}-P_j$ paths in $B_j$ such that $|E(P_{j,2})| - |E(P_{j,1})|\in\{1,2\}$. Let $C$ be the cycle consisting of the union of all paths $P_j$ and $P_{j,1}$ for $j \in \{0,\ldots,2k-1\}$. For each $B_j$, we can modify $C$ by replacing the path $P_{j,1}$ by $P_{j,2}$ which increases the length of the cycle by $x_j$ with $x_j \in \{1,2\}$. There are at least $k$ indices $j$ for which the values $x_j$ are the same. Thus we can obtain cycles of all lengths in $\{|E(C)|, |E(C)|+x,\ldots,|E(C)|+kx\}$ for some $x \in \{1,2\}$. Since $\gcd(k,x)=1$, this set contains all residues modulo k.
\end{proof}

If $k$ is even and $x=2$, then the set $\{|E(C)|, |E(C)|+x,\ldots,|E(C)|+kx\}$ at the end of the proof contains either all even or all odd residues modulo $k$. Thus, if $k$ is even we can only guarantee that at least half of all possible residues modulo $k$ occur in a $2k$-wiggly necklace.\\
To simplify the calculations in the main proof we introduce the following definition. 

\begin{definition}[$k$-good]
A graph $G$ is called \textbf{$k$-good} if it contains at least one of the following graphs as a subgraph: 
\begin{itemize}
    \item a $\theta (k,i)$-necklace with $i\in \{1,2\}$,
    \item a $k$-wiggly necklace, or
    \item a path and a cycle which are $k$-close.
\end{itemize}
\end{definition}

Note if a graph is $k$-good for some natural number $k$, then it is also $k'$-good for every natural number $k'$ with $k'\leq k$. It follows from Lemma~\ref{lem:theta_neck_pan_cyc}, Lemma~\ref{lem:kclose}, and Lemma~\ref{lem:wiggly_pan_cyc} that for every odd natural number $k$ there exists a natural number $f(k)$ such that every $f(k)$-good graph contains cycles of every length modulo $k$.

\begin{theorem}\label{thm:kgood}
Let $G$ be a graph and $k$, $m$ natural numbers with $k$ odd. If~$G$ is $162k^8$-good, then $G$ contains a cycle whose length is congruent to $m$ modulo $k$.
\end{theorem}
\begin{proof}
By Lemma~\ref{lem:wiggly_pan_cyc}, we can assume that $G$ contains a path and a cycle which are $162k^8$-close or a $\theta(162k^8,i)$-necklace where $i \in \{1,2\}$. Since $162k^8 > 3k^4$, by Lemma~\ref{lem:theta_neck_pan_cyc}, we can assume that $G$ contains a path and a cycle which are $162k^8$-close. By Lemma~\ref{lem:kclose}, the graph $G$ contains a $\theta(3k^4,i)$-necklace where $i \in \{1,2\}$. Now $G$ contains a cycle whose length is congruent to $m$ modulo $k$ by Lemma~\ref{lem:theta_neck_pan_cyc}.
\end{proof}

Thus, to prove Theorem~\ref{thm:main} it suffices to show that for any natural number $k$, every sufficiently large 3-connected cubic graph is $k$-good. 

\section{Paths containing given edges}

In this section we investigate the following question: Given a set $S$ of edges in a subcubic graph $G$, does there exist a path in $G$ containing $k$ edges of $S$? We show that such a path exists if $|S|$ is sufficiently large and $G$ is 2-connected. Note that 2-connectivity is a necessary constraint, since we cannot find such a path for $k=3$ if $G$ is a tree and $S$ contains all edges incident with vertices of degree 1.\\ 
We apply this result in Section 5 to construct cycles and paths which are $k$-close: If $C$ is a cycle such that $G-C$ has a block $B$ containing many vertices which have distance at most 2 from $C$, then we can find a path~$P$ containing $k$ of these vertices and $P$ is $k$-close to $C$.\\
Before we prove the main theorem of this section, we prove the special case where $G$ has a Hamiltonian path and $S$ consists of the edges which are not contained in the Hamiltonian path. Note that in this case we allow multiple edges in $G$ and we do not require $G$ to be 2-connected.

\begin{lemma}\label{lem:chord}
Let $k$ be a natural number, $G$ a subcubic multigraph with a Hamiltonian path~$H$, and $M = E(G)\setminus E(H)$. If $|M| \geq \frac{8}{3}k(k-1)+1$, then there exists a path $P$ in $G$ such that $|E(P)\cap M|\geq k$.
\end{lemma}
\begin{proof}
Let $H=p_1 \ldots p_n$. We define a total ordering on $V(G)$ by $p_i \leq p_j$ if and only if $i \leq j$. We write $u < v$ for $u,v\in V(G)$ if $u\leq v$ and $u\neq v$. Let $f(k)= \lceil \frac{8}{3}k(k-1)+1 \rceil$. We may assume $|M| = f(k)$ and $M = \{e_1,\ldots,e_{f(k)}\}$. For $i\in \{1,\ldots ,f(k)\}$, let $\ell_i$ and $r_i$ denote the end points of $e_i$ such that $\ell_i<r_i$. By possibly relabeling the edges in $M$, we may assume that $\ell_i < \ell_j$ whenever $i<j$.
Now consider the sequence $S=(r_1, r_2,\ldots,r_{f(k)})$. Since $|S|\geq  (\lceil \frac{8}{3}k \rceil-1)(k-1)+1$, by Theorem~\ref{thm:erdos-szekeres} the sequence $S$ has a decreasing subsequence of length $k$ or an increasing subsequence of length $\lceil \frac{8}{3}k \rceil$. If $S$ has a decreasing subsequence $S_{>}$ of length $k$, then it is easy to see that there exists a path containing the $k$ edges of $M$ which are incident with a vertex in $S_>$. Thus, we may assume that $S$ contains an increasing subsequence $S_<$ of length $\lceil \frac{8}{3}k \rceil$.
Let $M'\subset M$ be the set of edges in $M$ which are incident to a vertex in $S_{<}$. We say two edges $e,f \in M'$ are \emph{non-crossing} if $r(e)<\ell(f)$ or $r(f) < \ell(e)$. For $e\in M'$, let 
$$R(e) = \{e'\in M' \,|\, r(e)<\ell(e')\}\,.$$ 
We define a set $E' = \{e_1',\ldots ,e_m'\}$ of edges in the following way. Let $e_1'$ be the edge in $M'$ for which $\ell(e_1')$ is minimal. Suppose $e_i'$ is defined and $R(e_i')$ is non-empty, then we define $e_{i+1}'$ as the edge in $R(e_i')$ for which $\ell (e_{i+1}')$ is minimal. We continue like this until we reach an index $m$ such that $R(e_m')=\emptyset$. Note that the edges in $E'$ are pairwise non-crossing and there exists a path in $G$ containing all edges in $E'$, see Figure~\ref{fig:lem_chord}. Thus, we may assume $m<k$. For $e\in E'$ we define
$$C(e) = \{e' \in M' \, | \,  \ell(e)\leq \ell(e')<r(e)\}\,.$$
By construction of $E'$, the sets $C(e_1'),\ldots ,C(e_m')$ form a partition of $M'$. For $e\in E'$, let $H(e)$ denote the shortest subpath of $H$ containing the ends of all edges in $C(e)$. For each $i\in \{1,\ldots ,m\}$, we now construct a path $P_i$ in $H(e_i')\cup C(e_i')$ containing at least $\frac{3}{4} |C(e_i')|$ of the edges in $C(e_i')$ and whose endvertices are the endvertices of $H(e_i')$. Let $C(e_i') = \{f_1, \ldots , f_t\}$ with $\ell (f_1) < \cdots < \ell (f_t)$. Note that $f_1 = e_i'$. Since $C(e_i') \subseteq M'$, we have $r(f_1) < \cdots < r(f_t)$ and $H(e_i') = \ell (f_1)Hr(f_t)$. If $t=1$ or $t=2$, define $P_i = f_1$ or $P_i = f_1 + r(f_1)H\ell(f_2) + f_2$, respectively. If $t\geq 3$ is odd, then define
$$P_i =
f_1 + r(f_1)Hr(f_2) + f_2 + \ell(f_2)H\ell(f_3) + f_3 + \ldots + f_{t-1} + \ell(f_{t-1})H\ell(f_t) + f_t\,.
$$
If $t\geq 4$ is even, then we define
$$P_i =
f_1 + r(f_1)Hr(f_2) + f_2 + \ell(f_2)H\ell(f_3) + f_3 + \ldots + \ell(f_{t-2})H\ell(f_{t-1}) + f_{t-1} + r(f_{t-1})Hr(f_{t})\,.
$$
In this case $P_i$ contains all edges in $C(e_i')$ apart from $f_t$, while in the other cases $P_i$ contains all edges in $C(e_i')$. This implies $|E(P_i) \cap C(e_i')| \geq \frac{3}{4} |C(e_i')|$.\\
Note that if $e\in C(e_i')$, $f\in C(e_j')$, and $i+2\leq j$, then $e$ and $f$ are not crossing. Thus, the paths $P_i$ and $P_j$ are disjoint for $i,j\in \{1,\ldots ,m\}$ with $i+2\leq j$. Let $M_{\text{odd}}$ denote the union of all $C(e_i')$ where $i\in \{1,\ldots ,m\}$ is odd, and $M_{\text{even}} = M'\setminus M_{\text{odd}}$. One of $M_{\text{odd}}$ and $M_{\text{even}}$, say $M_{\text{odd}}$, contains least $\frac{4}{3}k$ edges. Let $P$ be a path which contains all the paths~$P_i$ with $i$ odd as a subpath. Now $P$ contains at least $\frac{3}{4}|M_{\text{odd}}| \geq k$ edges of~$M$.
\end{proof}

\begin{figure}[t]
    \centering
    \includegraphics[width=\textwidth]{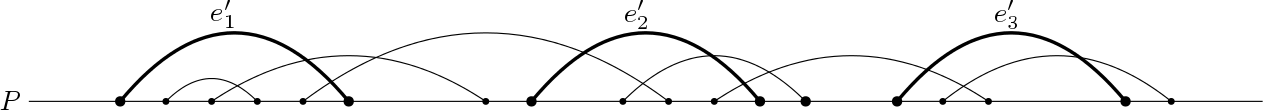}
    \caption{Proof of Lemma~\ref{lem:chord}}
    \label{fig:lem_chord}
\end{figure}

The lower bound for the size of $M$ in Lemma~\ref{lem:chord} is best possible up to a constant factor. To see this, let $K$ be a natural number and let $G_K$ be the graph consisting of a Hamiltonian path and a matching $M$ on $K^2$ edges as in Figure~\ref{fig:path_with_edges}. It is easy to see that every path in $G_K$ contains at most $3K-2$ edges of $M$. Thus, if there is a path containing at least $k$ edges of $M$, then $K>\frac{k}{3}$ and $|M|\geq \frac{k^2}{9}$.

We use Lemma~\ref{lem:chord} to prove a more general statement that holds for all subcubic 2-connected graphs.

\begin{figure}[h]
    \centering
    \includegraphics[width=1\textwidth]{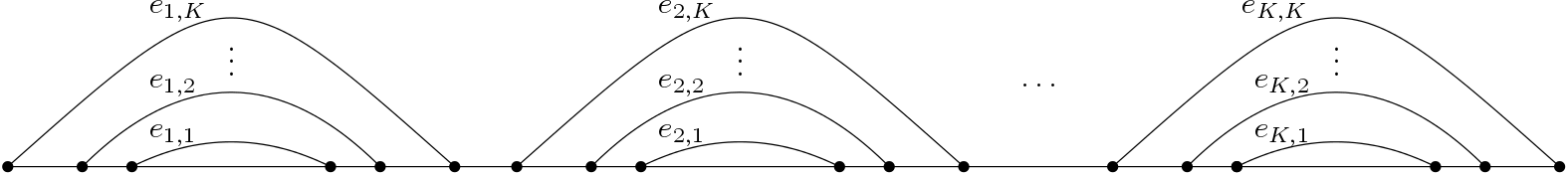}
    \caption{A graph where $|M|=K^2$ and $|E(P)\cap M|\leq 3K-2$ for every path $P$.}
    \label{fig:path_with_edges}
\end{figure}

\begin{theorem}\label{thm:chord_adv}
Let $G$ be a 2-connected subcubic graph, $S\subseteq E(G)$, and $k$ a natural number. If $|S| \geq 2^{4k^2}$, then there exists a path $P$ in $G$ such that $|E(P)\cap S| \geq k$.
\end{theorem}
\begin{proof}
Suppose the lemma is false and let $(G,S,k)$ be a counterexample where $|E(G)|$ is minimal. We may assume $k\geq 3$ since $G$ is 2-connected. We call the edges in $S$ special. If an edge $e$ is not special and not contained in a 2-edge-cut, then by minimality of $G$, there exists a path containing $k$ special edges in $G-e$, and thus also in $G$. Hence, we may assume that every edge which is not special is contained in a 2-edge-cut. \\
Suppose there exists a vertex $v$ of degree 2 which is not incident with a special edge. Let~$G'$ be the graph we obtain from $G$ by suppressing $v$. If $G'$ is simple, then $G'$ is a smaller counterexample. If $G'$ is not simple, then $G-v$ is 2-connected and $G-v$ is a smaller counterexample. Thus we may assume that every vertex of degree 2 is incident with at least one special edge.\\
Suppose $\{e,f\}$ is a non-trivial 2-edge-cut and $K_1$ and $K_2$ are the two components of $G-e-f$. Suppose $K_1$ does not contain a special edge. Then the graph $G'$ we obtain from~$G$ by contracting $K_1$ into a single vertex has fewer edges than $G$ but the same number of special edges. Since $G$ is a minimal counterexample, we can find a path containing $k$ special edges in $G'$ and therefore also in $G$. Thus, we may assume that every non-trivial 2-edge-cut separates the graph into two components which contain special edges. \\
Since $|S|\geq 2^{4k^2}$, the graph $G$ has at least $\frac{2}{3} \cdot 2^{4k^2}$ vertices. Let $v$ be any vertex in $G$. Since $G$ is subcubic, the distance class at distance $m$ away from $v$ contains at most $3 \cdot 2^{m-1}$ vertices. Let $d_{\text{max}}$ denote the maximum distance from a vertex to $v$ in $G$. We have
$$\frac{2}{3}\cdot 2^{3k^2} \leq |V(G)| \leq 1+\sum_{i=1}^{d_{\text{max}}}3 \cdot 2^{i-1} = \frac{3}{2}\cdot 2^{d_{\text{max}}+1}-2\,.$$  
Thus, we have $2^{d_{\text{max}}+1} \geq  2^{4k^2-2}$ and $d_{\text{max}} \geq 4k^2-3$. Since $G$ is 2-connected, this implies that $G$ has a cycle $C$ of length at least $8k^2-6$. Let $X$ be the set of vertices of $C$ that have degree 3 in $G$. We may assume that $C$ contains at most $k-1$ special edges. Since every vertex of degree 2 is incident with a special edge, $C$ contains at most $2(k-1)$ vertices of degree 2. Thus, $|X| \geq 8k^2 - 6 - 2(k-1) = 8k^2 - 2k - 4$.\\ 
Let $H=G-E(C)$. If $K$ is a component of~$H$, then either $|V(K)\cap X|\geq 2$ or $K$ is an isolated vertex and $V(K)\cap X = \emptyset$. By Lemma~\ref{lem:path_system}, we can choose in each component $K$ a collection of disjoint paths pairing up the vertices of $V(K) \cap X$ (apart from possibly one if $|V(K) \cap X|$ is odd). This defines a set of paths $\mathcal{P}$ with $|\mathcal{P}| \geq \frac{1}{3}|X| \geq \frac{1}{3}(8k^2- 2k - 4)$. Let $G'$ be an auxiliary graph we obtain by taking the cycle $C$ and adding a chord between the ends of each path in $\mathcal{P}$. In $G'$ the cycle $C$ has $|\PP|$ chords. Note that since $|\PP| \geq \frac{8}{3}k(k-1)+1$ for $k \geq 3$, there is a path in $G'$ containing at least $k$ chords of $C$ by Lemma~\ref{lem:chord}. This corresponds to a path $P$ in $G$ which uses edges of $C$ and at least $k$ paths in $\mathcal{P}$. Let $\mathcal{P}'\subseteq \mathcal{P}$ be the set of paths contained in $P$, so $|\PP'|\geq k$.\\
Let $Q$ be a path in $\mathcal{P'}$ which joins the vertices $u$ and $v$ of $C$, and let $e$ be the edge of $Q$ incident with $u$. If $e$ is not special, then there exists an edge $f$ such that $G-e-f$ is disconnected. Note that $f\in E(Q)$. Let $K$ be the component of $G-e-f$ not containing~$C$. If $K$ is an isolated vertex, then $e$ or $f$ is special. If $K$ is not an isolated vertex, then $K$ contains a special edge $s$. Since $Q$ contains $e$ and $f$ which separate $K$ from $C$, there is no other path in $\mathcal{P'}$ which contains a vertex of $K$. Since $G$ is 2-connected we can modify $P$ in $K$ so that it contains $s$. That is, we replace $Q$ by a path $Q'$ in $H$ which joins $u$ and $v$, contains $s$ and is disjoint from all paths in $\PP' \setminus \{Q\}$. By repeating this modification of $P$, we obtain a path in $G$ that contains at least one special edge for each path in $\PP'$. Since $|\PP'|\geq k$, this path contains at least $k$ special edges.
\end{proof}
We do not believe that the lower bound for $|S|$ in Theorem~\ref{thm:chord_adv} is optimal. However, Lang and Walther~\cite{langwalther} constructed a family of 2-connected cubic graphs where the length of a longest path in $G$ is $O(\log^2(|E(G)|))$. By choosing $S=E(G)$, this shows that the lower bound for $|S|$ in Theorem~\ref{thm:chord_adv} cannot be smaller than $\Omega (2^{\sqrt{k}})$. \\
We conclude this section by proving a more specialized lemma which we use in Section 5.

\begin{lemma}\label{lem:theta+matching}
Let $k$ be a natural number and $G$ a subcubic graph which is the union of a $\theta$-graph $\Theta = (P_1,P_2,P_3)$ and a matching $M$ such that every edge in $M$ has its endpoints on two different legs of $\Theta$. If $|M|\geq 3k^2$, then $G$ has a cycle $C$ such that $|E(C)\cap M|\geq k$.
\end{lemma}
\begin{proof}
For one pair of legs of $\Theta$, say $P_1$ and $P_2$, we have $|E(P_1,P_2)| \geq k^2$. Let $P_1=p_1p_2 \ldots p_n$ and $P_2=q_1q_2 \ldots q_m$ with $p_1=q_1$ and $p_n=q_m$. Define a total ordering of $V(P_1)$ by $p_i \leq p_j$ if and only if $i \leq j$ and a total ordering of $V(P_2)$ by $q_i \leq q_j$ if and only if $i \leq j$. Given an edge $e$ between $P_1$ and $P_2$ let $p_i(e)$ denote the end of $e$ on $P_i$ for $i \in \{1,2\}$. Let $e_1,\ldots,e_{k^2} \in M$ be distinct edges between $P_1$ and $P_2$ where $p_{1}(e_i) \leq p_{1}(e_{j})$ whenever $i \leq j$. By Theorem~\ref{thm:erdos-szekeres}, the sequence $S=(p_2(e_1),\ldots ,p_2(e_{k^2}))$ has a subsequence $S'=(p_2(e_1'),\ldots ,p_2(e_{k}'))$ of length $k$ which is increasing or decreasing. In each case, there is a cycle $C$ in $G$ containing all the edges $e_i'$ for $i \in \{1,\ldots ,k\}$, see Figure~\ref{fig:lem_3_4}.
\end{proof}

\begin{figure}[h]
    \centering
    \begin{subfigure}[b]{0.43\textwidth}
        \includegraphics[width=\textwidth]{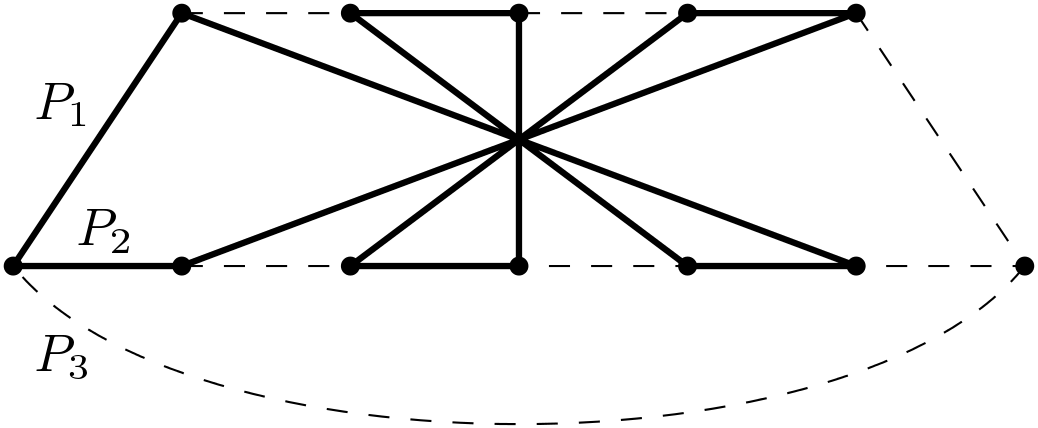}
        \caption{$S'$ is decreasing and $k$ odd}
    \end{subfigure}
    \hspace{.5cm}
    \begin{subfigure}[b]{0.5\textwidth}
        \includegraphics[width=\textwidth]{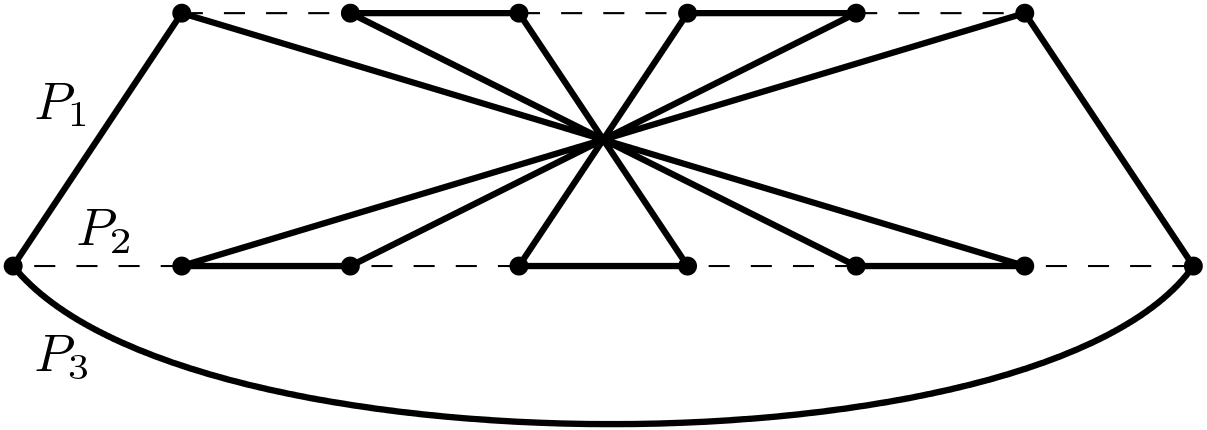}
        \caption{$S'$ is decreasing and $k$ even}
    \end{subfigure}
    \par\bigskip
    \begin{subfigure}[b]{0.45\textwidth}
        \includegraphics[width=\textwidth]{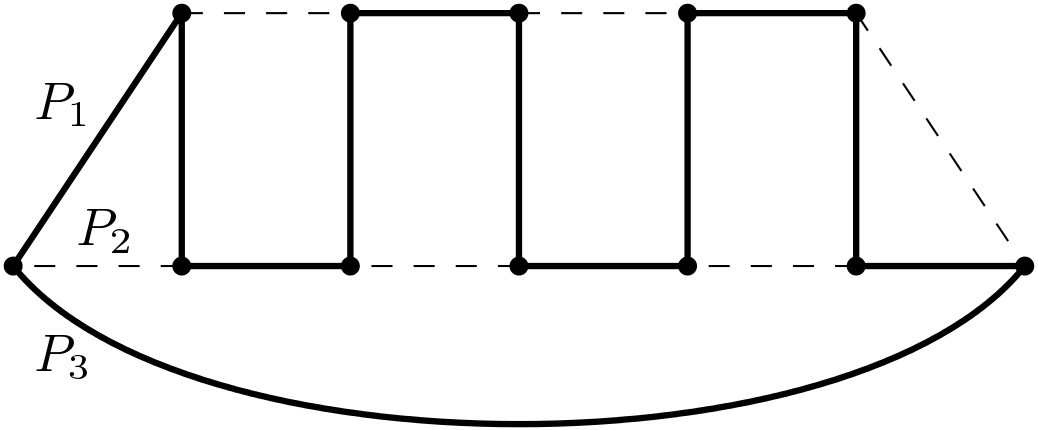}
        \caption{$S'$ is increasing}
    \end{subfigure}
    \caption{Proof of Lemma~\ref{lem:theta+matching}}
    \label{fig:lem_3_4}
\end{figure}

\section{Paths whose lengths differ by 1 or 2}

Let $G$ be a 2-connected subcubic graph and $x,y\in V(G)$. A result by Fan~\cite{fan} implies that if $d(v)=3$ for every $v\in V(G)\setminus\{x,y\}$, then there exist two $x-y$ paths whose lengths differ by 1 or 2. In this section we extend this result to the case where $V(G)\setminus \{x,y\}$ contains one or two vertices of degree 2. This will allow us in Section 5 to use non-trivial 3-edge-cuts for the construction of $k$-wiggly necklaces.

\begin{definition}[$f_C(x,y)$]
Let $C$ be a cycle in a graph and let $x,y$ be two distinct vertices of $C$. We define $f_C(x,y)$ as the absolute difference of the lengths of the two $x-y$ paths on~$C$.
\end{definition}

\begin{lemma} \label{lem:1degree2}
	Let $G$ be a 2-connected graph which is not a 3-cycle, $xy\in E(G)$ and $z \in V(G)\setminus\{x,y\}$. If $d_G(x)\leq 4$, $d_G(y)\leq 4$, $d_G(z)\leq 3$ and $d_G(v)=3$ for all $v \in V(G) \setminus \{x,y,z\}$, then there are two $x-y$ paths $P_1, P_2$ in $G-xy$ with $1\leq |E(P_1)|-|E(P_2)|\leq 2$.
\end{lemma}
\begin{proof}
Suppose the theorem is false and let $(G,x,y,z)$ be a counterexample where $|V(G)|+|E(G)|$ is minimum. Clearly $|V(G)| \geq 4$. Let $G'=G-xy$. Note that $d_{G'}(v)\geq 2$ for every $v\in V(G')\setminus\{x,y\}$. We may assume there are no two $x-y$ paths in $G'$ whose lengths differ by 1 or 2. \\
	
\textbf{Claim 1:} $G'$ is 2-connected. \\
Suppose $G'$ is not 2-connected. Since $G$ is not a cycle, there exists a 2-connected block $B$ in~$G'$. Let $Q_1$ be an $x-B$ path and let $Q_2$ be a $B-y$ path in $G'$. Note that $Q_1 \cap Q_2 = \emptyset$. Let $q_1$ and $q_2$ be the endvertices of $Q_1$ and $Q_2$ in $B$, respectively. Since $G$ is 2-connected, the block graph of $G'$ is a path, so the only vertices of degree 2 in $B$ are $q_1$, $q_2$, and possibly~$z$. Let $B' = B+q_1q_2$ if $q_1q_2\notin E(B)$ and $B'=B$ otherwise. If $B'$ is a triangle, then $B'=B$ and there are two $q_1-q_2$ paths $R_1$, $R_2$ in $B$ of lengths 1 and 2. If $B'$ is not a triangle, then by minimality of $G$ there are two $q_1-q_2$ paths $R_1$, $R_2$ in $B$ whose lengths differ by 1 or 2. Now $P_1=Q_1 \cup R_1 \cup Q_2$ and $P_2=Q_1 \cup R_2 \cup Q_2$ are two $x-y$ paths in $G'$ whose lengths differ by 1 or 2. \\

\textbf{Claim 2:} There are no non-trivial 2-edge-cuts in $G'$.\\
Suppose the claim is false, so there exists a non-trivial 2-edge-cut $\{e,f\}$ with $e,f \in E(G')$. Let $K$ be the component of $G'-e-f$ with $|V(K) \cap \{x,y,z\}| \leq 1$. We choose $e,f$ over all non-trivial 2-edge-cuts to minimise the order of $K$. By the choice of $e$ and $f$ the component $K$ is 2-connected. Let $e_K$ and $f_K$ denote the ends of $e$ and $f$ in $K$, respectively. We may assume $x\notin V(K)$.\\ 
First suppose $y \notin V(K)$. By Claim 1, $G'$ is 2-connected so there are two disjoint $\{x,y\} - \{e_K, f_K\}$ paths $Q_1$ and $Q_2$ in $G'$. If $K$ is a triangle, then let $P_1$, $P_2$ denote the $e_K-f_K$ paths in $K$ of lengths 1 and 2, respectively. If $K$ is not a triangle, then by minimality of $G$, there are two $e_K-f_K$ paths $P_1, P_2$ in $K$ whose lengths differ by 1 or~2. Now $Q_1\cup P_1\cup Q_2$ and $Q_1\cup P_2\cup Q_2$ are two $x-y$ paths in $G'$ whose lengths differ by 1 or 2.\\ 
Thus, we may assume $y \in V(K)$ and $z\notin V(K)$. We may assume $y\neq e_K$. If $K$ is a triangle, then let $P_1$, $P_2$ denote the $e_K-y$ paths in $K$ of lengths 1 and 2, respectively. If $K$ is not a triangle, then either $K$ or $K+e_Ky$ satisfies the conditions of Lemma~\ref{lem:1degree2}. By minimality of $G$ there are two $e_K-y$ paths $P_1$, $P_2$ in $K$ whose lengths differ by 1 or 2. Let $Q$ be an $x-e_K$ path in $G'$ having no edges in $K$. Now $Q\cup P_1$ and $Q\cup P_2$ are two $x-y$ paths in $G'$ whose lengths differ by 1 or 2. \\

Note that Claim~1 implies that if $u, v \in \{x,y,z\}$ and $d_G(u)=d_G(v)=2$, then $u$ and $v$ are non-adjacent. Let $G''$ denote the graph obtained from $G'$ by suppressing the vertices of degree 2. Note that $G''$ is cubic. If $G''$ is not simple, then Claim 2 implies that $G''$ consists of three parallel edges and $G'$ is a $\theta$-graph consisting of 4 or 5 vertices. In this case it easy to see that there are two $x-y$ paths in $G'$ whose lengths differ by 1 or 2. Thus we may assume that $G''$ is simple. By Claim 2, $G''$ is 3-connected, so by Theorem~\ref{thm:tutte} there exists an induced non-separating cycle $C$ in $G'$ containing $x$ and not containing $y$. Let $w\in V(C)\setminus\{z\}$ such that $f_C(x,w) \in \{1,2\}$. Let $P_1$ and $P_2$ denote the two $x-w$ paths in~$C$, and let $Q$ be a $w-y$ path intersecting $C$ only in $w$. Now $P_1\cup Q$ and $P_2\cup Q$ are two $x-y$ paths in $G'$ whose lengths differ by~1 or~2.
\end{proof}

From Lemma~\ref{lem:1degree2} we can immediately derive the following.

\begin{theorem}\label{thm:1degree2}
Let $G$ be a 2-connected subcubic graph and $x,y,z\in V(G)$. If $d(v) = 3$ for all $v\in V(G)\setminus\{x,y,z\}$, then there are two $x-y$ paths $P_1, P_2$ with $1\leq |E(P_1)|-|E(P_2)|\leq 2$.
\end{theorem}
\begin{proof}
The statement is trivial if $G$ is a 3-cycle. If $xy \in E(G)$ and $G$ is not a 3-cycle, then by Lemma~\ref{lem:1degree2} there exist two $x-y$ paths in $G-xy$ whose lengths differ by 1 or 2. If $xy \notin E(G)$, then $G'=G+xy$ satisfies the conditions of Lemma~\ref{lem:1degree2}. Thus there are two $x-y$ paths in $G$ whose lengths differ by 1 or 2.
\end{proof}

The next step is to allow two vertices of degree 2 in $V(G)\setminus \{x,y\}$. Note that in the following theorem we require the vertices $x_1, x_2, y,z$ to have degree exactly 2.

\begin{theorem} \label{thm:2deg2-ver1}
	Let $x_1, x_2, y, z$ be four distinct vertices of degree 2 in a 2-connected graph~$G$. If $d(v) = 3$ for all $v\in V(G)\setminus\{x_1, x_2, y ,z\}$, the vertices $x_1$, $x_2$ are not adjacent, and $x_1$, $x_2$ are not opposite vertices in a 4-cycle in $G$, then there are two $x_1-x_2$ paths $P_1$, $P_2$ with $|E(P_1)|-|E(P_2)|\in\{1,2\}$.
\end{theorem}

\begin{proof}
	Suppose the theorem is false and let $(G,x_1, x_2, y, z)$ be a counterexample where $G$ has minimum size. Clearly we can assume $|V(G)| \geq 5$. \\ 
	
	\textbf{Claim 1:} $x_1y, x_1z, x_2y, x_2z \notin E(G)$. \\
	Suppose the claim is false. We may assume $x_1y \in E(G)$. Let $x_1'$ be the neighbour of $x_1$ distinct from $y$ and let $y'$ be the neighbour of $y$ distinct from $x_1'$. Since $G$ is 2-connected and $x_1x_2\notin E(G)$, we have $x_1' \neq y'$ and $x_1'\neq x_2$.\\
	First suppose $x_1'y' \in E(G)$. In this case both $x_1'$ and $y'$ have degree 3 in $G$. We can assume that there are no two $x_1'-x_2$ paths in $G'=G-x_1-y$ whose lengths differ by 1 or 2, since otherwise there are also two $x_1-x_2$ paths in $G$ whose lengths differ by 1 or 2. Thus, by minimality of $G$, we have $x_1'x_2\in E(G')$ or $x_1'$, $x_2$ are contained in a 4-cycle. Similarly, we can assume that there are no two $y'-x_2$ paths in $G'$ whose lengths differ by 1 or 2. Again, minimality of $G$ implies that $y'x_2\in E(G')$ or $y'$, $x_2$ are contained in a 4-cycle. Note that $x_2$ cannot be adjacent to both $x_1'$ and $y'$ since otherwise $G$ is a 5-cycle with a chord and contains only three vertices of degree 2. Thus $x_2$ is contained in a 4-cycle with $x_1'$ and it follows by 2-connectivity of $G$ that $G$ is a 6-cycle with chord $x_1'y'$. It is easy to see that also in this case there are two $x_1-x_2$ paths whose lengths differ by 1 or 2. \\
	Now suppose $x_1'y' \notin E(G)$. Let $G''$ be the graph obtained from $G-x_1-y$ by adding the edge $e=x_1'y'$. Note that $G''$ is 2-connected and $x_2$, $z$ are the only vertices of degree 2 in~$G''$. By Theorem~\ref{thm:1degree2}, there are two $x_1'-x_2$ paths $Q_1$, $Q_2$ in $G''$ whose lengths differ by 1 or 2. If $Q_1$ does not contain $e$, let $P_1$ be the $x_1-x_2$ path consisting of $x_1x_1'$ and $Q_1$. If $Q_1$ contains $e$, let $P_1$ be the $x_1-x_2$ path we obtain from $Q_1$ by replacing $e$ with the path $x_1yy'$. We analogously define an $x_1-x_2$ path $P_2$ using $Q_2$. Note that $|E(P_1)|=|E(Q_1)|+1$ and  $|E(P_2)|=|E(Q_2)|+1$, so $P_1$ and $P_2$ are as desired. \\
	
    \textbf{Claim 2:} There are no non-trivial 2-edge-cuts in $G$.\\
	First, suppose there exists a non-trivial 2-edge-cut $\{e,f\}$ such that $G-e-f$ has a component $K$ with $|V(K)\cap \{x_1,x_2,y,z\}|\leq 1$. We choose such a 2-edge-cut $\{e,f\}$ for which the component $K$ containing at most one vertex of $\{x_1,x_2,y,z\}$ has minimal size. Note that~$K$ is 2-connected. Let $e_K,f_K$ denote the ends of $e$ and $f$ in $K$. If $V(K)\cap \{x_1,x_2\} = \emptyset$, then, since $G$ is 2-connected, there exist two disjoint $\{x_1,x_2\}-\{e_K,f_K\}$ paths $P_1,P_2$ in $G-E(K)$. Since $K$ is 2-connected, by Theorem~\ref{thm:1degree2} there are two $e_K-f_K$ paths $Q_1,Q_2$ in~$K$ whose lengths differ by 1 or 2. Now $P_1\cup Q_1\cup P_2$ and $P_1\cup Q_2\cup P_2$ are two $x_1-x_2$ paths whose lengths differ by 1 or 2. Thus, we may assume $x_1\in V(K)$. Again, by Theorem~\ref{thm:1degree2}, there are two $x_1-e_K$ paths $Q_1, Q_2$ whose lengths differ by 1 or 2. Let $P$ be an $e_K-x_2$ path in $G-E(K)$. Now $Q_1\cup P$ and $Q_2\cup P$ are two $x_1-x_2$ paths whose lengths differ by 1 or 2. Thus, we have the following:
	\begin{description}
	\item[(*)] For every non-trivial 2-edge-cut $\{e,f\}$ in $G$, each component of $G-e-f$ contains two vertices of $\{x_1,x_2,y,z\}$. 
	\end{description}
	Let $\{e,f\}$ be a non-trivial 2-edge-cut for which the component $K$ of $G-e-f$ containing~$x_1$ has minimal size. Now $K$ is 2-connected by (*), Claim 1, and since $x_1$ is not adjacent to~$x_2$. Let $e_K$ and $f_K$ denote the ends of $e$ and $f$ in $K$.
	Let $L$ be the component of $G-e-f$ different from $K$.
	First suppose $x_2\in V(K)$. By (*), we have $y\in V(L)$ and $z\in V(L)$. By minimality of $G$ there are two $x_1-x_2$ paths in $K$ whose lengths differ by 1 or 2. Thus we may assume $x_2\in V(L)$. By (*), we may assume $y\in V(K)$ and $z\in V(L)$. If $x_1$ is not adjacent to $e_K$ and $x_1$, $e_K$ are not opposite vertices in a 4-cycle, then by minimality of $G$ there are two $x_1-e_K$ paths in $K$ whose lengths differ by 1 or 2. Since these paths can be extended to $x_1-x_2$ paths in $G$, we may assume that $x_1$ is adjacent to $e_K$ or $x_1,e_K$ are opposite vertices in a 4-cycle.
	Similarly we can assume that either $x_1$ is adjacent to $f_K$ or $x_1$, $f_K$ are opposite vertices in a 4-cycle. If $x_1$ is adjacent to both $e_K$ and $f_K$, then $K$ is a 4-cycle where $x_1$ and $y$ are opposite vertices in a 4-cycle. In this case there exist $x_1-e_K$ paths in $K$ of lengths 1 and 3. Thus, we may assume that $x_1$ and $f_K$ are opposite vertices in a 4-cycle. If also $x_1$, $e_K$ are opposite vertices in a 4-cycle, then there are $x_1-e_K$ paths of length 2 and 4 in $K$. If $x_1$ is adjacent to $e_K$ then there are $x_1-e_K$ paths of length 1 and 3 in $K$. This concludes the proof of Claim 2.\\
	
\begin{figure}[t]
    \centering
    \includegraphics[width=.8\textwidth]{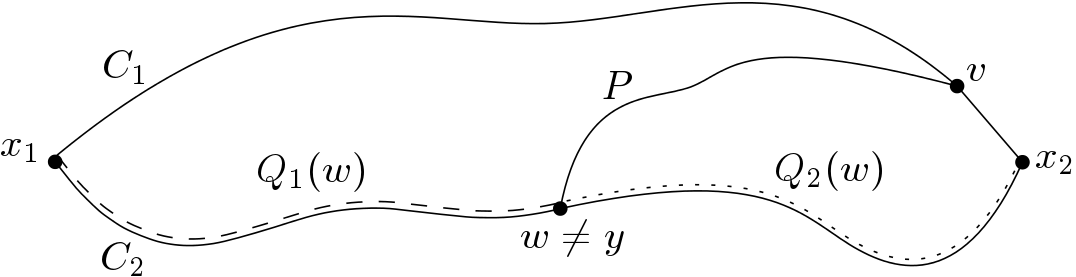}
    \caption{Proof of Theorem~\ref{thm:2deg2-ver1}}
    \label{fig:lem_4_6}
\end{figure}

Let $G'$ denote the graph obtained from $G$ by suppressing the vertices of degree 2. By Claim~2 and the fact that $G'$ is cubic, the graph $G'$ is simple and 3-connected. Now Theorem~\ref{thm:tutte} implies that there is a non-separating induced cycle $C$ in $G$ containing $x_1$ and not containing $z$.\\
Suppose $x_2 \notin V(C)$. There are two different vertices $v_1,v_2\in V(C)$ such that $f_C(x_1,v_1)=f_C(x_1,v_2)\in \{1,2\}$. In particular, there exists a vertex $v\in V(C)$ different from $y$ for which $f_C(x_1,v)\in\{1,2\}$. Let $P$ be a $v-x_2$ path in $G-E(C)$, and let $Q_1,Q_2$ be the two $x_1-v$ paths on $C$. Now $Q_1\cup P$ and $Q_2\cup P$ are two $x_1-x_2$ paths whose lengths differ by 1 or 2.\\
Thus we may assume $x_2 \in V(C)$. 
Let $C_1$ and $C_2$ be the two $x_1-x_2$ paths on $C$. We may assume $|E(C_1)|\leq |E(C_2)|$. If $|E(C_1)| = |E(C_2)|$, we may assume that $C_2$ does not contain~$y$.
Let $v$ be the neighbour of $x_2$ on $C_1$. Note that $v\in V(C)\setminus\{x_1,x_2,y\}$. 
For a vertex $w\in V(C_2)$, let $Q_1(w) = x_1C_2w$ and $Q_2(w) = wC_2x_2$, see Figure~\ref{fig:lem_4_6}. Moreover, for $w\in V(C_2)$, let 
$$f(w) = |E(C_1)|-2+|E(Q_2(w))|-|E(Q_1(w))|\,.$$ 
Note that as we move from $x_1$ to $x_2$ along $C_2$, the function $f$ decreases by 2 at every vertex.
We have $f(x_1) = |E(C)|-2$. Since $|E(C)|\geq 5$, we have $f(x_1)\geq 3$.\\
By definition, $f(x_2)=|E(C_1)|-|E(C_2)|-2\leq -2$. If $f(x_2)<-2$, then there are two vertices $w_1,w_2\in V(C_2)\setminus \{x_1,x_2\}$ such that $|f(w_1)|=|f(w_2)|\in \{1,2\}$. In particular, there exists a vertex $w\in V(C_2)\setminus \{x_1,x_2,y\}$ with $|f(w)|\in \{1,2\}$. 
If $f(x_2)=-2$, then $|E(C_1)|=|E(C_2)|$ and $y$ is not contained in $C_2$. Also in this case there exists a vertex $w\in V(C_2)\setminus \{x_1,x_2,y\}$ with $|f(w)|\in \{1,2\}$.\\
In each case, we can choose $w\in V(C_2)$ such that $d(w)=3$ and $|f(w)|\in \{1,2\}$. Let $P$ be a $v-w$ path in $G-E(C)$. Now $(C_1-vx_2)\cup P\cup Q_2(w)$ and $Q_1(w)\cup P\cup \{vx_2\}$ are two $x_1-x_2$ paths and the difference of their lengths is $|E(C_1)|-2+|E(Q_2(w))|-|E(Q_1(w))|=|f(w)|$, which is 1 or 2 by our choice of $w$.
\end{proof}

\begin{figure}[b]
    \centering
    \includegraphics[width=0.85\textwidth]{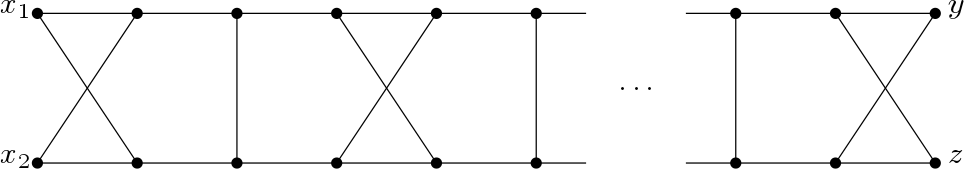}
    \caption{A graph where no two $x_1,x_2$-paths differ by one or two in length.}
    \label{fig:no_4_cyc_cond_counter}
\end{figure}

The graph $G$ in Figure~\ref{fig:no_4_cyc_cond_counter} shows that Theorem~\ref{thm:2deg2-ver1} is not true if we allow $x_1$ and $x_2$ to be opposite vertices in a 4-cycle $C$. Let $G'=G-C$ and let $x_1', x_2'$ be the two vertices in $N(C)$. There are no two $x_1'-x_2'$ paths in $G'$ whose lengths differ by 1 or 2 which shows that Theorem~\ref{thm:2deg2-ver1} is not true if we allow $x_1$ and $x_2$ to be adjacent. Note that the graph in Figure~\ref{fig:no_4_cyc_cond_counter} contains two disjoint $\{x_1,x_2\}-\{y,z\}$ paths which are $k$-close provided that $dist(x_1,y)\geq k$. The following lemma shows that if $dist(x_1,y)\geq k$ and $dist(x_2,y)\geq k$, then we can always find two disjoint $\{x_1,x_2\}-\{y,z\}$ paths which are $k$-close or two $x_1-x_2$ paths whose lengths differ by 1 or 2. These $k$-close paths can be used in Section 5 to find a cycle and a path which are $k$-close.

\begin{lemma}\label{lem:2deg2-ver2} 
Let $k$ be a natural number, $G$ a 2-connected graph and $x_1,x_2,y, z\in V(G)$ distinct vertices of degree 2. Assume that $dist(x_i,y)\geq k$ for $i=1,2$. If $d(v)= 3$ for all $v\in V(G)\setminus\{x_1,x_2,y, z\}$, then $G$ contains two $x_1-x_2$ paths $P_1,P_2$ with $|E(P_1)|-|E(P_2)|\in\{1,2\}$, or $G$ contains two disjoint $\{x_1,x_2\}-\{y,z\}$ paths $Q_1,Q_2$ such that $|E(Q_1,Q_2)|\geq k$.
\end{lemma} 
\begin{proof}
We prove the statement by induction on $k$. If $x_1$ and $x_2$ are not adjacent and not opposite vertices in a 4-cycle, then the existence of the two desired $x_1-x_2$ paths follows immediately from Theorem~\ref{thm:2deg2-ver1}. Thus we may assume that $x_1$ and $x_2$ are adjacent or opposite vertices in a 4-cycle. We can easily find two disjoint $\{x_1,x_2\}-\{y,z\}$ paths $Q_1$ and $Q_2$ with $|E(Q_1,Q_2)|\geq 1$ so the statement is true for $k\leq 1$. Suppose now $k\geq 2$ and that the statement is true for $k-1$ and $k-2$. If $x_1$ and $x_2$ are adjacent, we define~$x_1'$ to be the neighbour of $x_1$ different from $x_2$, unless this vertex is $z$, in which case we choose the neighbour of $z$ different from $x_1$ to be $x_1'$. We analogously define $x_2'$. We set $G'=G-x_1-x_2-z$ if $z$ was a neighbour of $x_1$ or $x_2$ and $G'=G-x_1-x_2$ otherwise.\\
If $x_1$ and $x_2$ are opposite vertices in a 4-cycle $C$, then let $u$ and $v$ be the two vertices in $N(C)$. If one of these vertices is $z$, say $u$, then we define~$x_1'$ as the neighbour of $z$ not in~$C$ and $x_2'$ as $v$, otherwise we choose $u$ and $v$ as $x_1'$ and $x_2'$. We set $G' = G-V(C)-z$ if $z$ had a neighbour on $C$ and $G' = G-V(C)$ otherwise.\\
It is easy to see that $G'$ is connected. Suppose $G'$ is not 2-connected. Then the block $B_1$ containing $x_1'$ is different from the block $B_2$ containing $x_2'$ and both blocks are endblocks. We may assume that $B_1$ contains at most one of $y$ and $z$. Let $c$ be the cutvertex of $G'$ which is contained in $B_1$. By Theorem~\ref{thm:1degree2}, there are two $x_1'-c$ paths $P_1$ and $P_2$ in $B_1$ whose lengths differ by 1 or 2. Let $P$ be a $c-x_2'$ path in $G'$. Now $P_1\cup P$ and $P_2\cup P$ are two $x_1'-x_2'$ paths whose lengths differ by 1 or 2 and it is easy to see how they can be extended to $x_1-x_2$ paths with the same property. Hence, we may assume that $G'$ is 2-connected.\\
If $z\notin V(G')$, then by Theorem~\ref{thm:1degree2} there exist two $x_1'-x_2'$ paths in $G'$ whose lengths differ by 1 or 2 and they can be extended to $x_1-x_2$ paths with this property. Thus, we may assume that $z\in V(G')$. Note that if $x_1$ and~$x_2$ are adjacent, then dist$(x_1',y)\geq$ dist$(x_1,y)-1\geq (k-1)$ and dist$(x_2',y)\geq (k-1)$. If $x_1$ and $x_2$ are opposite vertices in a 4-cycle, then dist$(x_1',y)\geq $ dist$(x_1,y)-2\geq (k-2)$ and dist$(x_2',y)\geq (k-2)$. We may assume that there are no two $x_1'-x_2'$ paths in $G'$ whose lengths differ by 1 or 2. By induction there exist two $\{x_1',x_2'\} -\{y,z\}$ paths $Q_1'$ and $Q_2'$ such that $|E(Q_1',Q_2')|\geq k-1$ (if $x_1$ and $x_2$ are adjacent in $G$) or $|E(Q_1',Q_2')|\geq k-2$ (if $x_1$ and $x_2$ are opposite vertices in a 4-cycle). The paths $Q_1'$ and $Q_2'$ can be extended to $\{x_1,x_2\} -\{y,z\}$ paths $Q_1$ and $Q_2$. If $x_1$ and~$x_2$ were adjacent, then $|E(Q_1,Q_2)| = |E(Q_1',Q_2')| + 1 \geq k$. If they were opposite vertices in a 4-cycle, then $|E(Q_1,Q_2)| = |E(Q_1',Q_2')| + 2 \geq k$. In any case, $|E(Q_1,Q_2)|\geq k$.
\end{proof}

\section{Proof of the main theorem}

In this section we prove the main theorem of this paper. The general idea is to show that if $G$ is large enough, then $G$ is $k$-good. We begin the proof by taking a minimal $u,v$-$\theta$-graph $\Theta$ in $G$ where $u$ and $v$ are far apart in $G$. We would like to investigate the 2-connected endblocks of $G-\Theta$, but unfortunately $G-\Theta$ might contain many vertices of degree~1. Instead we investigate the subgraph $H$ of $G-\Theta$ which is induced by the vertices of degree at least~2 in $G-\Theta$. Now there might be vertices of degree less than~2 in $H$, but we show in Section 5.1 that we may assume there are only few of them. We distinguish two types of 2-connected endblocks of $H$ depending on whether its neighbours on $\Theta$ lie on only one leg or on at least two different legs. We call these endblocks $\Theta$-isolated and $\Theta$-connecting, respectively. In Section~5.2 we show that $G$ is $k$-good if the number of $\Theta$-connecting endblocks in $H$ is sufficiently large. In Section~5.3 we show the same if $H$ has many $\Theta$-isolated endblocks. Finally, the only remaining case is where $H$ has only few 2-connected endblocks. This is dealt with in Section~5.4 which concludes the proof.

\subsection{General framework and initial observations}

\begin{definition}[short $\theta$-graph]
A \textbf{shortest $u,v$-$\theta$-graph} in $G$ is a $u,v$-$\theta$-graph $\Theta = (P_1,P_2,P_3)$ for which $|E(\Theta)|$ is minimal.
We say a graph is a \textbf{short $\theta$-graph} if it is a shortest $u,v$-$\theta$-graph for some vertices $u,v$. 
\end{definition}

Note that by Menger's theorem a shortest $u,v$-$\theta$-graph in $G$ exists if and only if there exists no 2-cut separating $u$ and $v$. In particular, it always exists if $G$ is 3-connected.

\begin{definition}[$F(\Theta)$, $\Theta$-friendly]
Let $\Theta = (P_1,P_2,P_3)$ be a short $\theta$-graph in a 3-connected cubic graph $G$. We write $F(\Theta)$ for the set of all vertices in $G-\Theta$ with at least two neighbours in $\Theta$.
We say a vertex $v\in V(G)$ is \textbf{$\Theta$-friendly} if $v\in F(\Theta )$ and $N_{\Theta} (v)\not\subseteq V(P_i)$ for every $i\in\{1,2,3\}$. 
\end{definition}

Every vertex $v\in F(\Theta)$ is either $\Theta$-friendly or $N_{\Theta}(v)$ is contained in one leg of $\Theta$. If $v\in F(\Theta)$ is not $\Theta$-friendly, then any two neighbours of $v$ in $\Theta$ have distance at most 2 on $\Theta$ since the legs have minimal length. The following shows that we can assume that $F(\Theta)$ has only few vertices which are $\Theta$-friendly or whose neighbours have distance 1 on $\Theta$.

\begin{lemma} \label{lem:easy_specialcases}
Let $\Theta = (P_1,P_2,P_3)$ be a short $\theta$-graph in a 3-connected cubic graph $G$. If 
\begin{description}
    \item[(a)] $G-E(\Theta)$ has at least $\frac{3}{2}k$  edges with both ends in $V(\Theta)$, or
    \item[(b)] $G-\Theta$ has at least $\frac{3}{2}k$ vertices which are $\Theta$-friendly, or
    \item[(c)] $G-\Theta$ has at least $3k$ isolated vertices, or
    \item[(d)] $G-\Theta$ has at least $3k$ isolated edges, or
    \item[(e)] $\Theta$ has at least $\frac{3}{2}k$ edges which are contained in a triangle in $G$,
\end{description}
then $G$ is $k$-good.
\end{lemma}
\begin{proof}
\textbf{(a)} Since $\Theta$ is a short $\theta$-graph, all the legs are induced paths in $G$. If $G-E(\Theta)$ has at least $\frac{3}{2}k$ edges with both ends in $V(\Theta)$, then there is a leg, say $P_1$, such that $k$ of the edges are incident with $P_1$. Now $P_1$ and the cycle $P_2 \cup P_3$ are $k$-close.
\par
\textbf{(b)} If $G-\Theta$ has at least $\frac{3}{2}k$ vertices which are $\Theta$-friendly, then there exists a leg of $\Theta$, say $P_1$, such that at least $k$ vertices of $G-\Theta$ have one neighbour on $P_1$ and one neighbour on a different leg. Now $P_1$ and the cycle $P_2 \cup P_3$ are $k$-close.
\par
\textbf{(c)} Let $S$ denote the set of isolated vertices in $G-\Theta$ which are not $\Theta$-friendly. By~(b), we may assume $|S| \geq \frac{3}{2}k$. Since $\Theta$ is a short $\theta$-graph, the neighbours of $v$ on $\Theta$ induce a path of length 2 for every $v\in S$. Now there are two legs of $\Theta$, say $P_1$ and $P_2$, such that at least $k$ vertices in $S$ have their neighbours on $P_1$ or $P_2$. These $k$ vertices and their incident edges together with $P_1$ and $P_2$ form a $\theta (k,1)$-necklace.
\par
\textbf{(d)} We colour some vertices in $G-\Theta$ with colours 1,2,3 such that $v\in V(G)$ has colour~$i$ if $v$ is contained in an isolated edge in $G-\Theta$ and $N_\Theta (v)\subset V(P_i)$. Let $X$ denote the set of isolated edges in $G-\Theta$ in which both vertices received a colour. It is easy to see that since $\Theta$ is a short $\theta$-graph and $G$ is 3-connected, no two adjacent vertices are coloured the same. Since there are at least $3k$ isolated edges in $G-\Theta$, we may assume by (b) that $|X| \geq 3k- \frac{3}{2}k=\frac{3}{2}k$. Thus $|V(X)| \geq 3k$, so at least $k$ vertices in $V(X)$ have the same colour, say colour 1. Let $Y \subset V(X)$ denote the set of vertices coloured 1. Since $\Theta$ is a short $\theta$-graph, the neighbours on $P_1$ of any vertex in $Y$ induce a path of length at most~2 in $P_1$. Hence there is a path $P'$ contained in the subgraph of $G$ induced by $P_1$ and $Y$, containing all vertices in $Y$. Since each vertex in $Y$ has a neighbour in colour 2 or 3, the path $P'$ and the cycle $P_2 \cup P_3$ are $k$-close.
\par
\textbf{(e)}  In this case there are two legs of $\Theta$, say $P_1$ and $P_2$, such that $P_1\cup P_2$ contains~$k$ edges which are contained in triangles $T_1,\ldots ,T_k$ in $G$. Now the union of $P_1$, $P_2$ and $T_1,\ldots ,T_k$ is a $k$-wiggly necklace. 
\end{proof}

It is possible that $F(\Theta)$ contains many vertices which are not $\Theta$-friendly. Unfortunately such vertices make it more difficult to show that $G$ is $k$-good. In the following we will therefore be interested in subgraphs of $G- (\Theta \cup F(\Theta ))$.

\begin{definition}[$B^+$, $\Theta^+$]
Let $\Theta$ be a short $\theta$-graph in a 3-connected cubic graph $G$ and let $B$ be a connected subgraph of $G-\Theta$. We define $B^+$ as the subgraph of $G$ induced by $B\cup N_{F(\Theta )}(B)$ and $\Theta^+$ as the subgraph of $G$ induced by $\Theta \cup F(\Theta)$.
\end{definition}

Note that $(B^+)^+ = B^+$ for every connected subgraph $B$ of $G-\Theta$. From now on we focus on the structure of $G-\Theta^+$. We start by showing that $G$ is $k$-good if $G-\Theta^+$ contains many vertices of degree less than 2. This implies that $G$ is $k$-good if $G-\Theta^+$ has many endblocks which are not 2-connected.

\begin{theorem}\label{thm:leaves}
Let $\Theta = (P_1,P_2,P_3)$ be a short $\theta$-graph in a 3-connected cubic graph $G$. If $G-\Theta^+$ has at least $8k$ vertices of degree at most 1, then $G$ is $k$-good.
\end{theorem}
\begin{proof}
Let $L$ denote the set of vertices in $G-\Theta^+$ which have degree at most 1, so $|L|\geq 8k$. Every vertex in $L$ has either at least two neighbours in $F(\Theta )$ or exactly one neighbour in each of $F(\Theta )$ and $\Theta$. Let $L'\subseteq L$ denote the set of vertices having no $\theta$-friendly neighbour. By Lemma~\ref{lem:easy_specialcases}~(b) we may assume $|L'|\geq |L|-\frac{3}{2}k > 6k$. For $i \in \{1,2,3\}$, let $F_i$ be the set of vertices in $F(\Theta )$ with two neighbours on $P_i$. For $i, j\in \{1,2,3\}$ with $i\neq j$, let $X_{i,j}$ be the set of vertices in $L'$ having a neighbour in both $F_i$ and $F_j$. Let $Y_{i,j}$ be the set of vertices in $L'$ having a neighbour in $F_i$ and in $P_j$. Finally, let $Z_i$ be the set of vertices in $L'$ having at least two neighbours in $F_i$ or exactly one in both $F_i$ and $P_i$. Let
\begin{align*}
    X & = X_{1,2} \cup X_{2,3} \cup X_{1,3}\\
    Y & = Y_{1,2} \cup Y_{1,3} \cup Y_{2,1} \cup Y_{2,3} \cup Y_{3,1} \cup Y_{3,2} \\
    Z & = Z_{1} \cup Z_{2} \cup Z_{3}
\end{align*}
Clearly, $X\cup Y \cup Z = L'$, so $|X| + |Y| + |Z| > 6k$. We distinguish three cases.\\

\textbf{Case 1:} $|X|\geq \frac{3}{2}k$.\\
We may assume that $|X_{1,2}\cup X_{1,3}|\geq k$. Let $F_i'\subseteq F_i$ be the set of vertices that have a neighbour in $X_{1,2}\cup X_{1,3}$. Since $\Theta$ is a short $\theta$-graph, for $i\in \{1,2,3\}$ there exists a path $P_i'$ in $G[V(P_i)\cup F_i']$ which contains $F_i'$. We can combine the paths $P_2'$ and $P_3'$ to a cycle $C$ in $\Theta^+$ which contains $F_2'\cup F_3'$ but no interior vertices of $P_1$. Now $P_1'$ and $C$ are $k$-close since every $x\in X_{1,2}\cup X_{1,3}$ has a neighbour in both $P_1'$ and $C$.\\

\textbf{Case 2:} $|Y|\geq 3k$.\\
We may assume that $|Y_{1,2}\cup Y_{1,3}|\geq k$.
For every $x\in Y_{1,2}\cup Y_{1,3}$, let $y(x)\in F_1$ be a neighbour of $x$. Since $\Theta$ is a short $\theta$-graph, there exists a path $P_1'$ in $G[V(P_1)\cup F_1]$ which contains $y(x)$ for every $x\in Y_{1,2}\cup Y_{1,3}$. Now $P_1'$ and the cycle $P_2\cup P_3$ are $k$-close.\\

\textbf{Case 3:} $|Z|\geq \frac{3}{2}k$.\\ 
We may assume that $|Z_{1}\cup Z_{2}|\geq k$.
For every $x\in Z_i$, let $P_x$ denote the shortest subpath of $P_i$ containing $N_{P_i}(x)$ and the neighbours of $N_{F_i}(x)$ on $P_i$. Since $\Theta$ is a short $\theta$-graph, the length of $P_x$ is at most 4. Note that this implies that $x$ has at most two neighbours in $F_i$. Let $G_x$ be the subgraph of $G$ induced by $P_x$, $N_{F_i}(x)$, and $x$. We define a subgraph $H_x$ of $G_x$ in the following way.
If $G_x$ contains a cycle $C_x$ of length 3 or 4 such that $|E(C_x\cap P_x)| = 1$, then we set $H_x = C_x$. It is easy to check that~$G_x$ contains such a cycle unless $P_x$ has length 3 and $x$ has two neighbours in $F_i$, as shown in Figure~\ref{fig:thm_5_5}(b). In this case, we set $H_x = G_x$. Note that if $H_x = G_x$, then $x$ and the two endvertices of $P_x$ form a 3-cut in $G$. It easy to see that in each case the endvertices of $H_x\cap P_x$ can be joined by two paths in $H_x$ whose lengths differ by 1 or 2. Moreover, the graphs $H_x$ are pairwise disjoint. Hence, the graph formed by the union of $P_1\cup P_2$ and all the graphs $H_x$ with $x\in Z_1 \cup Z_2$ is a $k$-wiggly necklace.
\end{proof}

\begin{figure}[H]
    \centering
    \begin{subfigure}[b]{0.31\textwidth}
        \includegraphics[width=\textwidth]{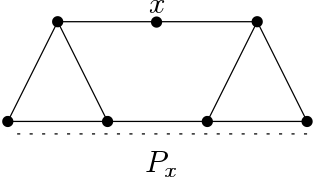}
        \caption{$|E(P_x)| \in \{3,4\}$}
        \label{fig:thm_5_5_1}
    \end{subfigure}
    \hspace{2 cm}
    \begin{subfigure}[b]{0.31\textwidth}
        \includegraphics[width=\textwidth]{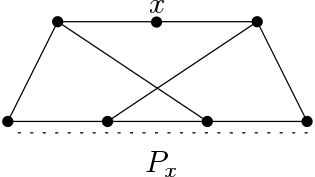}
        \caption{$|E(P_x)| =3$}
        \label{fig:thm_5_5_2}
    \end{subfigure}
    \par\bigskip
    \begin{subfigure}[b]{0.31\textwidth}
    \centering
        \includegraphics[width=.7\textwidth]{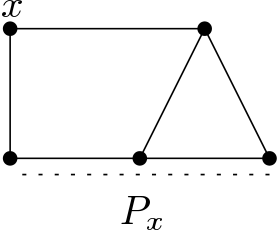}
        \caption{$|E(P_x)| \in \{2,3\}$}
        \label{fig:thm_5_5_3}
    \end{subfigure}
    \hspace{2 cm}
    \begin{subfigure}[b]{0.31\textwidth}
    \centering
        \includegraphics[width=.7\textwidth]{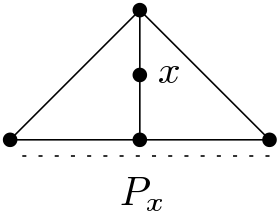}
        \caption{$|E(P_x)| =2$}
        \label{fig:thm_5_5_4}
    \end{subfigure}
    \caption{Proof of Theorem~\ref{thm:leaves}. The four ways $x \in Z$ can be joined to $\Theta^+$.}
    \label{fig:thm_5_5}
\end{figure}

\begin{definition}[$\Theta$-isolated, $\Theta$-connecting]
We say a connected subgraph $B$ of $G-\Theta^+$ is \textbf{$\Theta$-connecting} if $N_{\Theta} (B^+) \not\subseteq V(P_i)$ for every $i\in \{1,2,3\}$. We say $B$ is \textbf{$\Theta$-isolated} if $N_{\Theta}(B^+)\neq \emptyset$ and $B$ is not $\Theta$-connecting.
\end{definition}

Note that each 2-connected endblock of $G-\Theta^+$ is either $\Theta$-connecting or $\Theta$-isolated.

\subsection{Many $\Theta$-connecting 2-connected endblocks in $G-\Theta^+$}

Given a subgraph $H$ of $G-\Theta^+$, we are often interested in how it attaches to $\Theta$. We are typically interested in $N_{\Theta}(H^+)$, but we only want to include at most one neighbour for each vertex in $H^+-H$. For this purpose we define $\Theta$-projections which map vertices in $N_H(\Theta^+)$ to neighbours in $\Theta$ or vertices at distance 2 in $\Theta$.

\begin{definition}[$\Theta$-projection $\pi$]
Let $\Theta$ be a short $\theta$-graph in $G$ and $H = G-\Theta^+$. A function $\pi : N_H(\Theta^+) \rightarrow V(\Theta)$ is called a \textbf{$\Theta$-projection} if $\pi (v)$ is a vertex in $\Theta$ whose distance to $v$ in $G$ is minimal.
\end{definition}

By definition, if $v\in G-\Theta^+$ has a neighbour $u$ in $\Theta$, then $\pi (v) = u$. If $v\in G-\Theta^+$ has a neighbour in $\Theta^+$ but not in $\Theta$, then there exists a vertex $w\in F(\Theta )$ which is adjacent to both $v$ and $\pi (v)$. We now show that $G$ is $k$-good if $G-\Theta^+$ contains sufficiently many $\Theta$-connecting 2-connected endblocks.

\begin{theorem}\label{thm:theta-connecting}
Let $\Theta = (P_1,P_2,P_3)$ be a short $\theta$-graph in a cubic 3-connected graph $G$ and $H=G-\Theta^+$. If the number of $\Theta$-connecting 2-connected endblocks in $H$ is at least $21k^2+\frac{3}{2}k$, then $G$ is $k$-good.
\end{theorem}

\begin{proof} 
Let $S$ be the set of vertices which are contained in a $\Theta$-connecting 2-connected endblock of $H$ and which have a neighbour in $\Theta^+$. We colour the vertices in $S$ with colours 0, 1, 2, and 3 in the following way: If for some $i\in \{1,2,3\}$, the vertex $v$ has a neighbour in~$P_i$, or $v$ has a neighbour in $F(\Theta)$ which has two neighbours in $P_i$, then $v$ receives colour~$i$. Otherwise, $v$ receives colour 0. Notice that $v$ receives colour 0 if and only if $v$ has a neighbour in $F(\Theta)$ which is $\Theta$-friendly.
Let $\B$ be the set of $\Theta$-connecting 2-connected endblocks in $H$ which contain no vertices in colour 0. Notice that the blocks in $\B$ are pairwise disjoint since $G$ is cubic. By Lemma~\ref{lem:easy_specialcases} (b) we may assume that at most $\frac{3}{2}k$ vertices are coloured 0 and thus $|\B|\geq 21k^2$. Notice that each block in $\B$ contains vertices in at least two different colours since the blocks in $\B$ are $\Theta$-connecting. Let $\B_0$ be the sets of blocks in $\B$ which contain no two vertices of the same colour. For $i\in \{1,2,3\}$, let $\B_i$ be the set of blocks in $\B$ which contain at least two vertices in colour $i$. Notice that $\B_0$ is disjoint from $\B_1\cup \B_2 \cup \B_3$ and thus $|\B_0|+|\B_1\cup \B_2 \cup \B_3|\geq 21k^2$.  Let $\pi : N_H(\Theta^+) \rightarrow V(\Theta)$ be a $\Theta$-projection. We distinguish the following two cases.\\

\textbf{Case 1:} $|\B_1\cup \B_2 \cup \B_3|\geq 8k^2$.\\
We may assume $|\B_1|\geq \frac{8}{3}k^2$. For each $B\in \B_1$, let $u_B$ and $v_B$ denote two vertices coloured~1 in $B$. Let $G_1$ be the graph obtained from $P_1$ by adding the edges $\pi (u_B)\pi (v_B)$ for every $B\in \B_1$. Finally, let $P_B$ be a $\pi (u_B)-\pi (v_B)$ path whose interior vertices lie in $B^+$ and which contains a coloured vertex in $B$ whose colour is not 1 (such a path exists since $B$ is 2-connected). See Figure~\ref{fig:thm_5_10_case_1_and_case_2_1_1}(a) for an illustration. By Lemma~\ref{lem:chord}, there exists a path $P'$ in $G_1$ containing at least $k$ edges of the form $\pi (u_B)\pi (v_B)$. We obtain a path $P$ in $G$ from $P'$ by replacing each edge of the form $\pi (u_B)\pi (v_B)$ by $P_B$. Now $P$ and the cycle $P_2\cup P_3$ are $k$-close.\\

\textbf{Case 2:} $|\B_0|\geq 13k^2$.\\
Each block in $\B_0$ contains at most three coloured vertices, and at most one cutvertex in $H$. In particular, each block in $\B_0$ contains at most four vertices of degree 2. Let $\B_4$ denote the set of blocks in $\B_0$ containing four vertices of degree 2, and $\B' = \B_0\setminus \B_4$. Notice that by 3-connectivity of $G$, each block in $\B'$ contains precisely three vertices of degree 2 and at least two of them are coloured. We have $|\B_4|+|\B'|\geq 13k^2$.\\

\textbf{Case 2.1:} $|\B_4|\geq 10k^2$.\\
Each $B\in \B_4$ contains precisely one vertex of each colour and precisely one vertex $x_B$ which is a cutvertex in $H$. Let $y_B,z_B\in B$ be coloured vertices such that $y_B$ has colour $1$ and $z_B$ has colour 2. We define $Q_B$ as a $\pi (y_B)-x_B$ path which contains $z_B$ and whose interior vertices lie in $B^+$. Let $X=\{x_B:B\in \B_4\}$. We define $X_1\subset X$ as the subset of vertices which lie in a component of $H$ containing only one vertex of $X$. Let $X_2 = X\setminus X_1$. Note that $|X_1|+|X_2|\geq 10k^2$. We consider two cases.\\

\textbf{Case 2.1.1:} $|X_1|\geq 8k^2$.\\
For every $x_B\in X_1$, let $P_B$ be an $x_B-\Theta$ path with $P_B \cap B = \{x_B\}$ and let $x'_B$ be the endvertex of $P_B$ on $\Theta$, see Figure~\ref{fig:thm_5_10_case_1_and_case_2_1_1}(b). We may assume that at least $\frac{8}{3}k^2$ vertices of the form $x'_B$ lie on~$P_1$. Let $G_1$ be the graph obtained from $P_1$ by adding the edges $\pi (y_B)x'_B$ for every $B\in \B_4$ with $x_B\in X_1$ and $x_B'\in V(P_1)$. By Lemma~\ref{lem:chord}, there exists a path $P'$ in $G_1$ containing at least $k$ edges of the form $\pi (y_B)x_B'$. We obtain a path $P$ in $G$ from $P'$ by replacing each edge of the form $\pi (y_B)x'_B$ by $Q_B\cup P_B$. Now $P$ and the cycle $P_2\cup P_3$ are $k$-close.\\

\begin{figure}[]
    \centering
    \begin{subfigure}[b]{0.4\textwidth}
        \includegraphics[width=\textwidth]{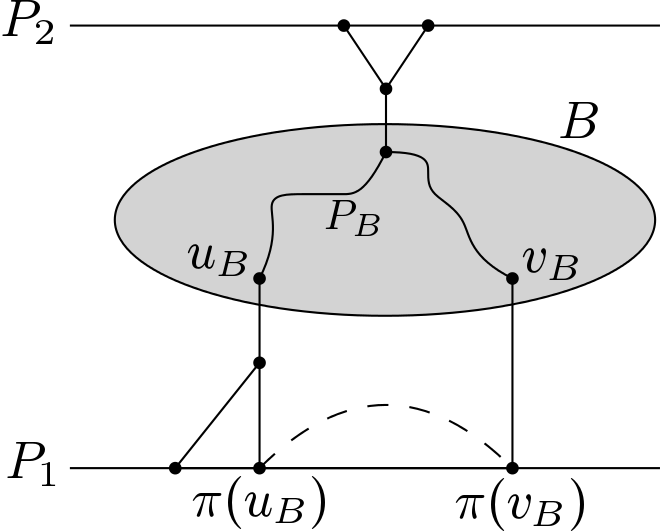}
        \caption{Case 1}
    \end{subfigure}
    \hspace{1cm}
    \begin{subfigure}[b]{0.5\textwidth}
        \includegraphics[width=\textwidth]{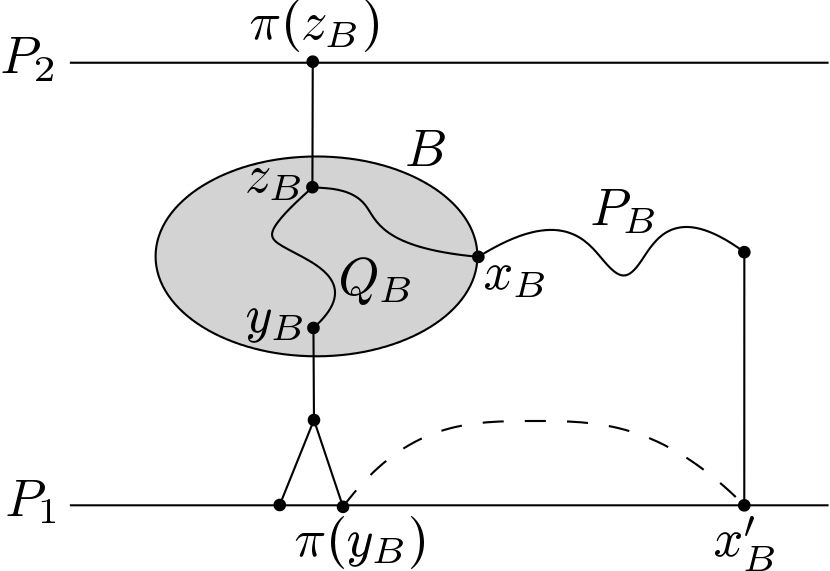}
        \caption{Case 2.1.1}
    \end{subfigure}
    \caption{Proof of Theorem~\ref{thm:theta-connecting}}
    \label{fig:thm_5_10_case_1_and_case_2_1_1}
\end{figure}

\textbf{Case 2.1.2:} $|X_2|\geq 2k^2$.\\
Let $\PP$ be a collection of disjoint paths in $H$ such that for every $P\in \PP$ the endvertices of~$P$ are in $X_2$ and $|\PP|$ is maximal. By Lemma~\ref{lem:path_system}, in every component of $H$ there is at most one vertex of $X_2$ which is not an endvertex of a path in $\PP$. This implies $|\PP|\geq \frac{1}{3}|X_2|\geq \frac{2}{3}k^2$. For any two blocks $B,B'\in \B_4$ which are joined by an $x_B-x_{B'}$ path $P\in \PP$, let $Q_{B,B'}$ be the path $Q_B\cup P\cup Q_{B'}$. Let $G_1$ be the graph obtained from $P_1$ by adding the edges $\pi (y_B)\pi (y_{B'})$ for every pair of blocks $B,B'\in \B_4$ which are joined by a path in $\PP$, see Figure~\ref{fig:thm_5_10_case_2_1_2_and_case_2_2}(a). Since $|\PP| \geq \frac{2}{3}k^2\geq \frac{8}{3}\cdot \frac{k}{2}(\frac{k}{2}-1)+1$, by Lemma~\ref{lem:chord} there exists a path $P'$ in $G_1$ containing at least $\frac{k}{2}$ edges of the form $\pi (y_B)\pi (y_{B'})$. We obtain a path $P$ in $G$ from $P'$ by replacing each edge of the form $\pi (y_B)\pi (y_{B'})$ by $Q_{B,B'}$. Since the path $Q_{B,B'}$ contains $z_B$ and $z_{B'}$, the path $P$ and the cycle formed by $P_2$ and $P_3$ are $k$-close.\\

\textbf{Case 2.2:} $|\B'|\geq 3k^2$.\\
For each $B\in \B'$, let $u_B$ and $v_B$ denote two distinct coloured vertices in $B$. Let $P_B$ be a $\pi (u_B)-\pi (v_B)$ path whose interior vertices lie in $B^+$ and let $G'$ be the graph obtained from $\Theta$ by adding the edges $\pi (u_B)\pi (v_B)$ for every $B\in \B'$, see Figure~\ref{fig:thm_5_10_case_2_1_2_and_case_2_2}(b). By Lemma~\ref{lem:theta+matching}, there exists a cycle $C'$ in $G'$ containing at least $k$ edges of the form $\pi (u_B)\pi (v_B)$. 
Let $\B''\subset \B'$ be the set of blocks $B$ for which $\pi (u_B)\pi (v_B) \in E(C')$. 
Let $C$ be the cycle in $G$ which is obtained from $C'$ by replacing $\pi (u_B)\pi (v_B)$ by $P_B$ for every $B\in \B''$. By Theorem~\ref{thm:1degree2}, every $B\in\B''$ contains two $u_B-v_B$ paths whose lengths differ by 1 or 2. Thus, the $\B''$-necklace formed by the union of $C$ and the blocks in $\B''$ is $k$-wiggly.
\end{proof}

\begin{figure}[]
    \centering
    \begin{subfigure}[b]{0.58\textwidth}
        \includegraphics[width=\textwidth]{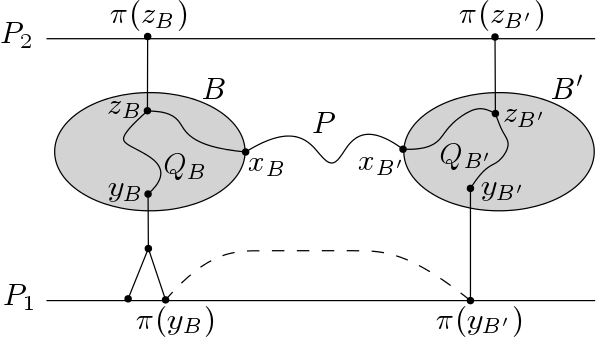}
        \caption{Case 2.1.2}
        \label{fig:thm_5_10_case_2_1_2}
    \end{subfigure}
    \hspace{1cm}
    \begin{subfigure}[b]{0.34\textwidth}
        \includegraphics[width=\textwidth]{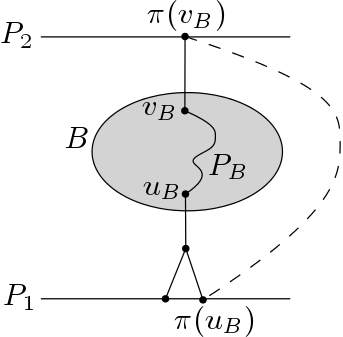}
        \caption{Case 2.2}
        \label{fig:thm_5_10_case_2_2}
    \end{subfigure}
    \caption{Proof of Theorem~\ref{thm:theta-connecting}}
    \label{fig:thm_5_10_case_2_1_2_and_case_2_2}
\end{figure}

\subsection{Many $\Theta$-isolated 2-connected endblocks in $G-\Theta^+$}

We now focus on $\Theta$-isolated subgraphs $H$ in $G-\Theta^+$. Typically $H$ is an endblock or a connected component of $G-\Theta^+$, so $|N_{\Theta}(H^+)|\geq 2$. 

\begin{definition}[$\Theta$-span]
Let $B$ be a $\Theta$-isolated subgraph of $G-\Theta^+$ with $N_{\Theta }(B^+)\subset V(P_i)$. Let $Q_B$ be the shortest subpath of $P_i$ containing $N_{\Theta}(B^+)$. The \textbf{$\Theta$-span} of $B$, denoted $\Sp(B)$, is defined as the vertex set of $Q_B$. 
\end{definition}

The $\Theta$-span is very useful for studying the interplay between different $\Theta$-isolated subgraphs. If there are many $\Theta$-isolated endblocks whose $\Theta$-spans are pairwise disjoint, then it is easy to find a long $\theta$-necklace or a $k$-wiggly necklace. If there are many $\Theta$-isolated endblocks whose $\Theta$-spans have pairwise non-empty intersection, then we distinguish two cases depending on whether many of them are pairwise crossing according to the following definition.

\begin{definition}[crossing subgraphs]
Let $B_1$ and $B_2$ be two $\Theta$-isolated connected subgraphs of $G-\Theta^+$ with $N_{\Theta}(B_1^+\cup B_2^+)\subset V(P_i)$. We say $B_1$ and $B_2$ are \textbf{crossing} if there exist vertices $x_1,y_1\in N_{\Theta} (B_1^+)$ and $x_2,y_2\in N_{\Theta}(B_2^+)$ such that $x_2\in x_1P_iy_1$, $y_1\in x_2P_iy_2$.
\end{definition}

If $B_1$, $B_2$ are two disjoint $\Theta$-isolated endblocks with $\Sp (B_1) \cap \Sp(B_2) \neq \emptyset$ and $B_1$, $B_2$ are not crossing, then one of the two $\Theta$-spans is contained in the other. This motivates the following definition.

\begin{definition}[$<_{\Theta}$]
Let $B_1$ and $B_2$ be two disjoint $\Theta$-isolated subgraphs of $G-\Theta^+$. We write $B_1 <_{\Theta} B_2$ if and only if $B_1$ and $B_2$ are not crossing and $\Sp(B_1)\subset \Sp(B_2)$.
\end{definition}

It is easy to see that two disjoint $\Theta$-isolated subgraphs $B_1,B_2$ of $G-\Theta^+$ satisfy $B_1<_{\Theta}B_2$ if and only if $\Sp(B_1)\subset \Sp(B_2)$ and $\Sp(B_1) \cap N_{\Theta}(B_2^+)=\emptyset$. 

\begin{definition}[$\Theta$-chain]
A \textbf{$\Theta$-chain} $\C$ of length $n$ is a sequence $(H_1,H_2,\ldots ,H_n)$ of pairwise disjoint $\Theta$-isolated components of $G-\Theta^+$ with $H_i <_{\Theta} H_{i+1}$ for $i\in \{1,\ldots ,n-1\}$. We say $\C$ is \textbf{special} if each $H_i$ has only 2-connected endblocks.
\end{definition}

Note that if $B_1,B_2,B_3$ are pairwise disjoint and $B_1<_{\Theta}B_2$ and $B_2<_{\Theta}B_3$, then $\Sp(B_1) \cap N_{\Theta}(B_3^+) \subseteq \Sp(B_2) \cap N_{\Theta}(B_3^+) = \emptyset$ and thus $B_1<_{\Theta}B_3$.
This shows that every subsequence of a $\Theta$-chain is again a $\Theta$-chain.\\
We now prove that if $G$ contains a long special $\Theta$-chain, then $G$ is $k$-good. The proof consists of several cases depending on the structure of the components in the $\Theta$-chain. As a rough guideline, if many endblocks have only few vertices of degree 2, then we find a $k$-wiggly necklace by using the results from Section 4. If there are many endblocks with many vertices of degree 2, then we construct a $\theta (k,i)$-necklace with $i\in \{1,2\}$.

\begin{lemma}\label{lem:theta-chain}
Let $\Theta = (P_1,P_2,P_3)$ be a shortest $u,v$-$\theta$-graph in a cubic 3-connected graph~$G$. If $G$ contains a special $\Theta$-chain $\C$ of length $5k$, then $G$ is $k$-good.
\end{lemma}
\begin{proof}
Let $H=G-\Theta^+$.
We may assume $N_{\Theta}(C^+)\subset V(P_1)$ for every component $C$ in~$\C$. 
Let $C_0$ be the first component in $\C$ and $c$ a vertex in $\Sp(C_0) - N_{\Theta^+}(C_0)$ (such a vertex exists since otherwise the two endvertices of $\Sp(C_0)$ form a 2-cut in $G$).
Note that $c$ is contained in the $\Theta$-span of every component in $\C$.
Let $L=uP_1c$ and $R=cP_1v$. 
We define $\Sp_L(B) = \Sp (B) \cap V(L)$ and $\Sp_R(B) = \Sp (B) \cap V(R)$ for every $\Theta$-isolated subgraph $B$ of $H$ with $N_{\Theta}(B^+)\subset V(P_1)$. If $\Sp_L (B)\neq \emptyset$, let $\ell_1(B), \ell_2(B)$ denote the vertices of $\Sp_L (B)$ which are closest to $u$ and $v$ on $P_1$, respectively. Note that $\ell_1(B) = \ell_2(B)$ if and only if $\Sp_L(B)$ consists of a single vertex. If $\Sp_R (B)\neq \emptyset$, we similarly define $r_1(B), r_2(B)$ as the vertices of $\Sp_R (B)$ which are closest to $u$ and $v$, respectively.
We colour the vertices of $N_H(\Theta^+)$ with colours 1 and 2 so that~$v$ is coloured 1 if and only if $v\in N_H(L)$ or $v$ has a neighbour in $N_{F(\Theta)}(L)$.
We say an endblock $B$ of $H$ is unbalanced if $B$ contains three vertices in the same colour. We say a component of $H$ is unbalanced if it contains an unbalanced endblock. Finally, we say a block or component is balanced if it is not unbalanced.\\

\textbf{Case 1:} $\C$ contains at least $k$ unbalanced components.\\
Let $\C' = (C_1,\ldots ,C_k)$ be a subsequence of $\C$ consisting of unbalanced components. For $i\in \{1,\ldots ,k\}$, let $B_i$ denote an unbalanced endblock of $C_i$.\\
Suppose $B_i$ contains three vertices coloured 1. If there exists $x\in B_i$ and $u\in F(\Theta)$ with $N(u) = \{ \ell_1 (B_i), \ell_2 (B_i), x\}$, then $\ell_1 (B_i)$ and $\ell_2 (B_i)$ have distance at most 2. Since~$\Theta$ is a short $\theta$-graph, this implies $|\Sp_L(B_i)| \leq 3$. Now there is at most one vertex in $N_{\Theta}(B_i^+) \setminus\{\ell_1(B_i), \ell_2(B_i)\}$ while there are at least two vertices in $B_i\setminus\{x\}$ coloured 1, a contradiction. Thus, there exists a $\Theta$-projection $\pi : N_H(\Theta^+) \rightarrow V(\Theta)$ such that $\ell_1 (B_i),\ell_2 (B_i)\in \pi (N_{B_i}(\Theta^+))$.
Now let $x_i,y_i,z_i \in N_{B_{i}}(\Theta^+)$ be three vertices in colour 1 with $\pi (x_i) = \ell_1 (B_i)$ and $\pi (y_i) = \ell_2 (B_i)$. Note that we have $\pi (z_i) \in \Sp_L(B_i)$. Let $w_i$ denote the neighbour of $z_i$ in $\Theta^+$. Let $Q_i$ be an $\ell_1(B_i)-\ell_2(B_i)$ path which contains~$z_i$ and whose interior vertices lie in $B_i^+$ (such a path exists since $B_i$ is 2-connected). It is easy to see that there exists an $\ell_1(B_i)-\ell_2(B_i)$ path $R_i$ in $\Theta^+$ which contains $w_i$, see Figure~\ref{fig:lem_5_15_case_1}.\\
If $B_i$ contains three vertices in colour 2, then we similarly choose $x_i,y_i,z_i \in N_{B_{i}}(\Theta^+)$ such that $\pi (x_i) = r_1 (B_i)$ and $\pi (y_i) = r_2 (B_i)$. We define
$Q_i$ as an $r_1(B_i)-r_2(B_i)$ path which contains $z_i$ and whose interior vertices lie in $B_i^+$. We define $R_i$ as an $r_1(B_i)-r_2(B_i)$ path in $\Theta^+$ which contains $w_i$, the neighbour of $z_i$ in $\Theta^+$.\\ 
For each $i\in \{1,\ldots ,k\}$, let $\Theta_i$ be the union of $Q_i$, $R_i$, and the edge $w_iz_i$. Note that $\Theta_i$ is a $\theta$-graph and $V(\Theta_i) \subset V(C_i^+)\cup \Sp_L(C_i) \cup \Sp_R(C_i)$. Since $\Sp_L(C)\cap \Sp_L(C') = \emptyset$ whenever $C$ and $C'$ are two different components of $\C$, we have that $\Theta_1,\ldots ,\Theta_k$ are pairwise disjoint. Now the union of $P_1\cup P_2$ and all subgraphs $\Theta_i$ with $i\in\{1,\ldots ,k\}$ contains a $\theta (k,1)$-necklace.\\

\begin{figure}[t]
    \centering
        \includegraphics[width=0.55\textwidth]{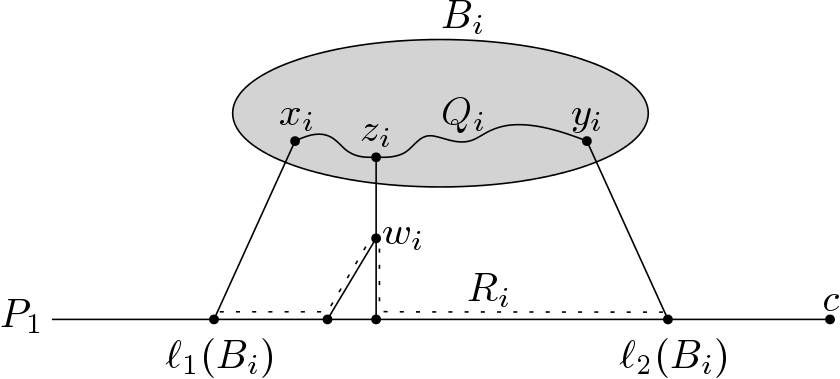}
    \caption{Proof of Lemma~\ref{lem:theta-chain} Case 1}
    \label{fig:lem_5_15_case_1}
\end{figure}

\textbf{Case 2:} $\C$ contains at least $4k$ balanced components.\\
Let $\C'$ be a subsequence of $\C$ of length $4k$ containing no unbalanced components. For an endblock $B$ of $H$, let $d_2(B)$ denote the number of vertices of degree 2 in $B$. By 3-connectivity of $G$ we have $d_2(B)\geq 3$ for every 2-connected endblock $B$. Note that the number of coloured vertices in a 2-connected endblock $B$ is $d_2(B)$ if $B$ is a component in~$H$. If $B$ contains a cutvertex in $H$, then $d_2(B)-1$ vertices in $B$ are coloured. Thus, if $d_2(B)\geq 6$ then $B$ is unbalanced. In particular, we have $d_2(B)\leq 5$ if $B$ is an endblock of a component in $\C'$. We define 
$$d_2(C) = \min \{d_2(B):\mbox{$B$ is an endblock of $C$}\}$$ 
for every component $C$ in $\C'$. For $i\in\{3,4,5\}$, we define $\C_i$ as the subsequence of $\C'$ containing all the components $C$ with $d_2(C)=i$. Note that $|\C_3|+|\C_4|+|\C_5|=|\C'|= 4k$.\\

\textbf{Case 2.1:} $|\C_3|\geq k$.\\
We may assume $|\C_3| = k$ and $\C_3 = (C_1,\ldots ,C_k)$. Let $B_i$ be an endblock of $C_i$ with $d_2(B_i)=3$ for every $i\in \{1,\ldots ,k\}$, and $\B = \{B_1,\ldots ,B_k\}$. Let $\pi : N_H(\Theta^+) \rightarrow V(\Theta)$ be a $\Theta$-projection. For $i\in\{1,\ldots ,k\}$, we define a path $Q_i$ in the following way. If $B_i$ has two vertices of the same colour, say $x_i$ and $y_i$, then let $Q_i$ be an $\pi(x_i)-\pi(y_i)$ path with interior vertices in $B_i^+$.\\ 
If $B_i$ contains no two vertices of the same colour, then it contains a vertex in each colour and a cutvertex in $C_i$. In this case, let $x_i$ be a coloured vertex in $C_i$ which is not contained in $B_i$. Let $y_i$ be the coloured vertex of $B_i$ which has the same colour as $x_i$. Now let $Q_i$ be an $\pi (x_i)-\pi (y_i)$ path with interior vertices in $C_i^+$.\\
Since $x_i$ and $y_i$ have the same colour, the vertices of $\pi(x_i)P_1\pi(y_i)$ are contained in $\Sp_L(C_i)$ or $\Sp_R(C_i)$. We have $\Sp_L(C_i) \cap \Sp_L(C_j) = \emptyset$ and $\Sp_R(C_i) \cap \Sp_R(C_j) = \emptyset$ for $i, j\in \{1, \ldots , k\}$ with $i\neq j$ since $\C_3$ is a special $\Theta$-chain.
Thus, the $k$ subpaths of the form $\pi(x_i)P_1\pi(y_i)$ with $i\in\{1,\ldots ,k\}$ are pairwise disjoint. Let $C$ be the cycle we obtain from $P_1\cup P_2$ by replacing the path $\pi(x_i)P_1\pi(y_i)$ by the path $Q_i$ for each $i\in\{1,\ldots ,k\}$. Let $N$ be the union of $C$ and all the blocks in $\B$. Clearly $N$ is a $\B$-necklace. Since each block in~$\B$ has only three vertices of degree 2, the necklace $N$ is $k$-wiggly by Theorem~\ref{thm:1degree2}.\\

\textbf{Case 2.2:} $|\C_4|\geq 2k$.\\
We may assume $|\C_4| = 2k$ and $\C_4 = (C_1,\ldots ,C_{2k})$. Let $B_i$ be an endblock of $C_i$ with $d_2(B_i)=4$ for every $i\in \{1,\ldots ,2k\}$, and $\B = \{B_{k+1},\ldots ,B_{2k}\}$. Let $\pi : N_H(\Theta^+) \rightarrow V(\Theta)$ be a $\Theta$-projection. Each $B_i$ contains at least three coloured vertices. Thus, $B_i$ contains two vertices $x_{i,1}$ and $x_{i,2}$ of the same colour and a third vertex $y_i$ of a different colour. Let~$z_i$ be the fourth vertex of degree 2 in $B_i$. Either $z_i$ is a cutvertex in $C_i$ or it has the same colour as $y_i$.\\ 
Let $S_i = \ell_2(B_i)P_1r_1(B_i)$ for $i\in \{1,\ldots ,2k\}$. Note that $N_{\Theta}(C_j^+) \subset S_i$ for $j\in\{1,\ldots ,i-1\}$ and $|N_{\Theta}(C_j^+)|\geq 4$. This implies $|E(S_i)|\geq 4(i-1)+1$. In particular, the distance between $\pi (x_{i,1})$ and $\pi (y_i)$ on $P_1$ is at least $4i-3$. Let $S_i'$ be a shortest $x_{i,1}-y_i$ path in $B_i$. Since~$\Theta$ is a short $\theta$-graph, we have $|E(S_i')|+4 \geq 4i-3$. Thus for $i\in \{k+1,\ldots 2k\}$ we have $|E(S_i')|\geq 4i-7\geq 4k-3 \geq k$. For the same reason, the distance between $x_{i,2}$ and $y_i$ in $B_i$ is at least $k$ for $i\in\{k+1,\ldots ,2k\}$.\\
Now we can apply Lemma~\ref{lem:2deg2-ver2} to each block in $\B$. Suppose for some $i\in\{k+1,\ldots ,2k\}$ there exist two disjoint $\{x_{i,2},x_{i,1}\}-\{y_i,z_i\}$ paths $Q_{i,1}$, $Q_{i,2}$ in $B_i$ with $|E(Q_{i,1},Q_{i,2})| \geq k$. It is easy to see that either $Q_{i,1}$ can be extended to a cycle which is $k$-close to $Q_{i,2}$, or $Q_{i,2}$ can be extended to a cycle which is $k$-close to $Q_{i,1}$, see Figure~\ref{fig:lem_5_15_case_2_2}. Thus, we can assume by Lemma~\ref{lem:2deg2-ver2} that for each $i\in\{k+1,\ldots ,2k\}$ there exist two $x_{i,1}-x_{i,2}$ paths $P_{i,1}$ and $P_{i,2}$ such that $|E(P_{i,1})|-|E(P_{i,2})|\in\{1,2\}$. As in Case 2.1, we can now find a cycle $C$ which contains all the paths $P_{i,1}$ for $i\in\{k+1,\ldots ,2k\}$. The union of $C$ with the blocks in $\B$  forms a $k$-wiggly $\B$-necklace.\\

\begin{figure}[]
    \centering
        \includegraphics[scale=.65]{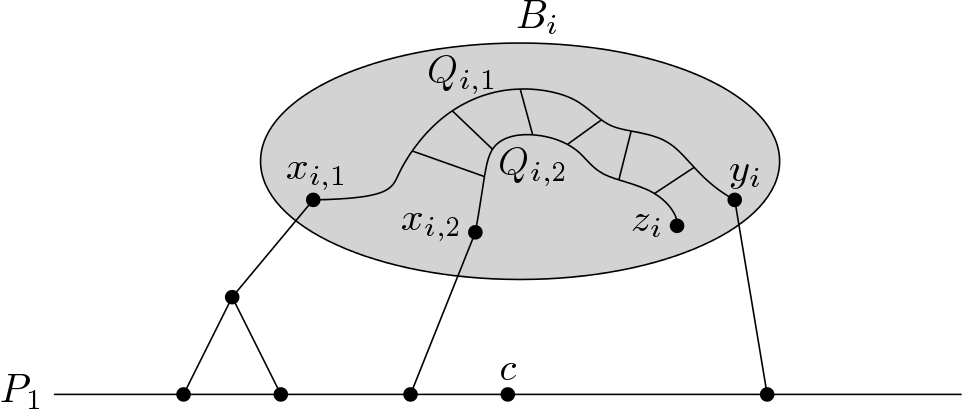}
    \caption{Proof of Lemma~\ref{lem:theta-chain} Case 2.2}
    \label{fig:lem_5_15_case_2_2}
\end{figure}

\textbf{Case 2.3:} $|\C_5|\geq k$.\\
We may assume $|\C_5| = k$ and $\C_5 = (C_1,\ldots ,C_k)$. If $B$ is an endblock of some component in~$\C_5$, then we have $5\leq d_2(B)$ by definition of $\C_5$ and $d_2(B)<6$ since $B$ is balanced. Thus, we have $d_2(B) = 5$ for every endblock $B$ of a component in $\C_5$. There are two vertices of each colour in $B$ and one cutvertex. In particular, the vertices $\ell_1(B),\ell_2(B),r_1(B),r_2(B)$ exist and are pairwise distinct. 
Let $\pi : N_H(\Theta^+) \rightarrow V(\Theta)$ be a $\Theta$-projection such that $r_1 (B)\in \pi (N_B(\Theta^+))$ for every endblock $B$ of a component in $\C_5$. It is possible that there exists a vertex in $B^+$ which is adjacent to both $r_1(B)$ and $r_2(B)$, in which case $r_1(B)$ and $r_2(B)$ have distance 2 on $\Theta$ and their common neighbour on $\Theta$, say $s(B)$, is contained in $\pi (N_B(\Theta^+))$. In this situation we choose $z(B)\in N_{B}(\Theta^+)$ such that $\pi (z(B)) = s(B)$. If $r_1(B)$ and $r_2(B)$ have no common neighbour in $B^+$, then we choose $z(B)\in N_{B}(\Theta^+)$ such that $\pi (z(B)) = r_1(B)$.\\
Let $w(B)$ be the neighbour of $z(B)$ in $\Theta^+$. Let $B_i$ and $B_i'$ be two different endblocks of~$C_i$ for $i\in\{1,\ldots ,k\}$. We may assume $r_1(B_i)\in cP_1r_1(B_i')$. Let $Q_i$ be an $r_1(B_i)-r_2(B_i')$ path whose interior vertices lie in $C_i^+$ and which contains $z(B_i')$ (such a path exists since $B_i'$ is 2-connected).
It is easy to see that there exists an $r_1(B_i)-r_2(B_i')$ path $R_i$ in $\Theta^+$ which contains $w(B_i')$. For each $i\in \{1,\ldots ,k\}$, let $\Theta_i$ be the union of $Q_i$, $R_i$, and the edge $w(B_i')z(B_i')$. As in Case 1, the $\theta$-graphs $\Theta_1$,\ldots,$\Theta_k$ are pairwise disjoint. Now the union of $P_1\cup P_2$ and $\Theta_1,\ldots,\Theta_k$ contains a $\theta (k,1)$-necklace.
\end{proof}

We now show that a graph with many $\Theta$-isolated 2-connected endblocks is $k$-good. We distinguish essentially three cases: there are many endblocks with pairwise disjoint spans, or many endblocks which are pairwise crossing, or a sequence of endblocks $B_1$,\ldots ,$B_\ell$ such that $B_i<_{\Theta}B_j$ for $i<j$. In the last case we use this sequence to construct a $\Theta$-chain which contains a special $\Theta$-chain as a subsequence.
   
\begin{theorem}\label{thm:theta-isolated}
Let $\Theta = (P_1,P_2,P_3)$ be a shortest $u,v$-$\theta$-graph in a cubic 3-connected graph $G$ and $H=G-\Theta^+$. If the number of $\Theta$-isolated 2-connected endblocks in $H$ is at least $5700k^6$, then $G$ is $k$-good.
\end{theorem}
\begin{proof}
We may assume $k\geq 3$ since $\Theta$ contains a path and a cycle which are 2-close. Let $n= 1900k^6$. We may assume that there exist $n$ pairwise disjoint $\Theta$-isolated 2-connected endblocks $B_1,\ldots ,B_n$ such that $N_{\Theta}(B_i^+)\subset V(P_1)$ for $i\in\{1,,\ldots ,n\}$. For two vertices $x,y\in V(P_1)$ we write $x\leq y$ if $x\in uP_1y$. Note that this defines a total order on $V(P_1)$. We write $x<y$ if $x\leq y$ and $x\neq y$.
For an endblock $B$ of $H$, let $\ell (B)$ and $r(B)$ denote the vertices of $\Sp (B)$ which are closest to $u$ and $v$ on $P_1$, respectively. Let $\pi : N_H(\Theta^+) \rightarrow V(\Theta)$ be a $\Theta$-projection such that $\{\ell (B),r(B)\}\subset \pi (N_B(\Theta^+))$ for each 2-connected endblock $B$ of $H$ for which $\ell (B)$ and $r(B)$ have no common neighbour in $F(\Theta)$. 
We will only use $\pi$ in this proof when we consider blocks $B$ that contain at least four vertices of degree 2. In this case it is easy to see that $\ell (B)$ and $r(B)$ have no common neighbour in $F(\Theta)$ since otherwise $\Sp (B)$ could contain at most three vertices.
Let $\ell_i = \ell (B_i)$ and $r_i = r(B_i)$. We may assume $\ell_i < \ell_{i+1}$ for $i\in\{1,\ldots ,n-1\}$. Since $n> 1896k^6 = 12k^3\cdot 158k^3$, by Theorem~\ref{thm:erdos-szekeres} the sequence $R=(r_1 ,\ldots ,r_n)$ has a strictly decreasing subsequence of length $158k^3$ or a strictly increasing subsequence of length $12k^3$.\\

\textbf{Case 1:} $R$ has a strictly increasing subsequence of length $12k^3$.\\
Let $B_1',\ldots ,B_{12k^3}'$ be the blocks corresponding to the increasing subsequence of $R$. Let $\ell_i' = \ell (B_i')$ and $r_i'= r(B_i')$. Thus we have $\ell_i' < \ell_{i+1}'$ and $r_i'<r_{i+1}'$ for $i\in\{1,\ldots ,12k^3-1\}$. We distinguish two cases.\\

\textbf{Case 1.1:} $\Sp(B_i') \cap \Sp (B_{i+6k^2}')\neq \emptyset$ for some $i\in\{1,\ldots ,12k^3-6k^2\}$.\\
Let $\B = \{B_i',\ldots ,B_{i+6k^2}'\}$. For $j\in \{i+1,\ldots ,i+6k^2-1\}$ we have $\ell_j' < \ell_{i+6k^2}' < r_i' < r_j'$, so $r_i'\in \Sp (B_j')$. In particular, $r_i'$ is contained in $\Sp (B)$ for every $B\in \B$. Let $G'$ be the graph we get by adding the edges $\ell (B)r(B)$ for every $B\in \B$ to $P_1\cup P_2$. Note that $G'$ is the cycle $P_1\cup P_2$ together with $6k^2+1$ chords which are pairwise crossing. It is easy to see that $G'$ has a cycle $C'$ which contains all edges of the form $\ell (B)r(B)$ apart from possibly one. Let $C$ be the cycle we get by replacing each edge $\ell (B)r(B)$ of $C'$ by an $\ell(B)-r(B)$ path with interior vertices in $B^+$. Let $\B_3\subseteq \B$ be the subset of blocks which contain precisely three vertices of degree 2. If $|\B_3|\geq k+1$, then by Theorem~\ref{thm:1degree2} the $\B_3$-necklace formed by the union of $C$ and the blocks in $\B_3$ is $k$-wiggly. Thus, we may assume $|\B_3|\leq k$.\\
Let $\B' = \B\setminus \B_3$. Note that $|\B'|\geq 6k^2-k+1$.
We have $|N_B(\Theta^+)|\geq 3$ for each $B\in \B'$ so there exist $x_B,y_B,z_B \in N_B(\Theta^+)$ such that $\pi(x_B) = \ell (B)$, $\pi (y_B) = r(B)$, and $\pi (z_B)\notin \{\ell (B), r(B)\}$. Let $B_0$ be the block in $\B'$ for which $\ell (B_0)$ is closest to $u$ on $P_1$. The vertices $\ell (B_0)$ and $r(B_0)$ split the cycle $P_1\cup P_2$ into two paths $Q_1$ and $Q_2$, see Figure~\ref{fig:thm_5_16_case_1_1_two}. We may assume that $Q_1$ contains $\ell (B)$ and $Q_2$ contains $r(B)$ for every $B\in\B'$. One of these two paths, say $Q_1$, contains at least $\frac{1}{2}(|\B'|-1)$ vertices of the form $\pi (z_B)$ with $B\in \B'\setminus\{B_0\}$. Let $Q'$ be the graph we get from $Q_1$ by adding the edges $\pi (x_B)\pi (z_B)$ for every $B\in \B'\setminus \{B_0\}$ with $\pi (z_B)\in V(Q_1)$. Note that 
$$\frac{1}{2}(|\B'|-1) \geq 3k^2-\frac{1}{2}k \geq \frac{8}{3}k(k-1)+1\,,$$ 
so by Lemma~\ref{lem:chord} there exists a path $P'$ in $Q'$ containing at least $k$ edges of the form $\pi (x_B)\pi (z_B)$. For each $B\in \B'$, let $P_B$ be a $\pi (x_B)-\pi (z_B)$ path which contains $y_B$ and whose interior vertices lie in $B^+$. By replacing each edge $\pi (x_B)\pi (z_B)$ in $P'$ by $P_B$, we obtain a path $P$ which is $k$-close to $Q_2$. Let $P_0$ be an $\ell (B_0)-r(B_0)$ path with interior vertices in $B_0^+$. Now $P$ and the cycle $P_0\cup Q_2$ are $k$-close.\\

\begin{figure}[]
        \centering
        \includegraphics[width=.7\textwidth]{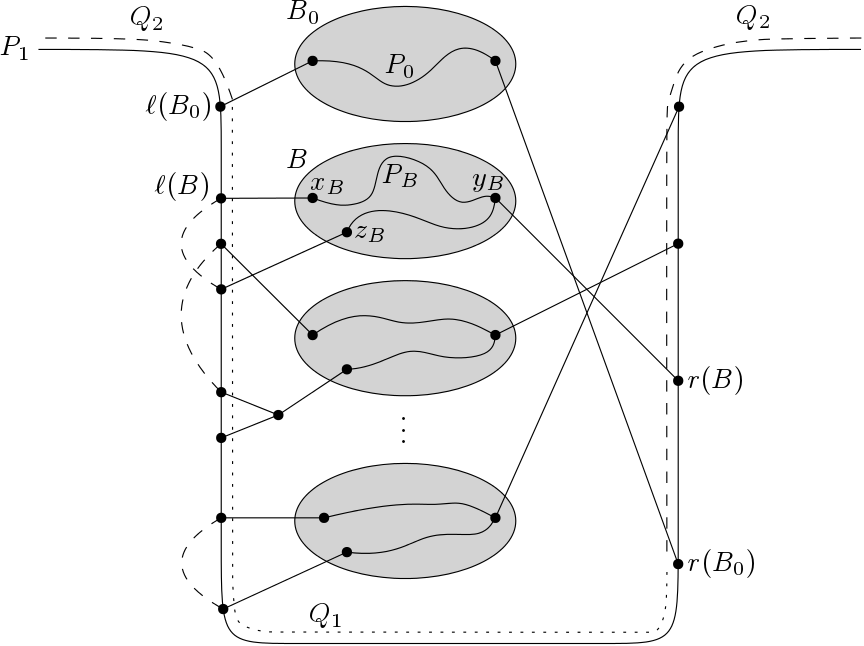}
        
    \caption{Proof of Theorem~\ref{thm:theta-isolated} Case 1.1}
    \label{fig:thm_5_16_case_1_1_two}
\end{figure}

\textbf{Case 1.2:} $\Sp(B_i') \cap \Sp (B_{i+6k^2}') = \emptyset$ for all $i\in\{1,\ldots ,12k^3-6k^2\}$.\\
For $i\in \{0,\ldots ,2k-1\}$, let $B_i'' = B_{1+6ik^2}'$ and $\B = \{B_0'',\ldots ,B_{2k-1}''\}$. Note that the $\Theta$-spans of the blocks in $\B$ are pairwise disjoint. Clearly $G$ has a cycle $C$ which goes through every block in $\B$. Let $\B_3\subseteq \B$ be the subset of blocks which contain precisely three vertices of degree 2. If $|\B_3|\geq k$, then by Theorem~\ref{thm:1degree2} the $\B_3$-necklace formed by the union of $C$ and the blocks in $\B_3$ is $k$-wiggly. Thus, we may assume $|\B_3|<k$. Let $\B' = \B\setminus \B_3$. Note that $|\B'| \geq k$. Let $x_B,y_B,z_B \in N_B(\Theta^+)$ such that $\pi(x_B) = \ell (B)$, $\pi (y_B) = r(B)$, and $\pi (z_B)\notin \{\ell (B), r(B)\}$ for each $B\in \B'$. Let $P_B$ be a $\pi (x_B)-\pi (y_B)$ path which contains $z_B$ and whose interior vertices lie in $B^+$, and let $w_B$ be the neighbour of $z_B$ in $\Theta^+$. It is easy to see that there exists a $\pi (x_B)-\pi(y_B)$ path $Q_B$ in $\Theta^+$ which contains $w_B$. Finally let $\Theta_B$ be the union of $P_B$, $Q_B$, and the edge $w_Bz_B$. Now the union of $P_1\cup P_2$ and all subgraphs~$\Theta_B$ with $B\in \B'$ contains a $\theta (k,1)$-necklace.\\

\textbf{Case 2:} $R$ has a strictly decreasing subsequence of length $158k^3$.\\
Note that if $\ell_i <\ell_j$ and $r_j<r_i$, then $\Sp (B_j)\subset \Sp (B_i)$. Thus, there exist $158k^3$ pairwise disjoint $\Theta$-isolated 2-connected endblocks $B_1',\ldots ,B_{158k^3}'$ such that $\Sp (B_{i}') \subset \Sp (B_{i+1}')$ for $i\in \{1,\ldots ,158k^3-1\}$. Let $C_i$ denote the component of $H$ containing $B_i'$ and let $\C = (C_1,\ldots ,C_{158k^3})$. The components in $\C$ are not necessarily distinct. Note that $158k^3\geq 156k^3+6k^2$ for $k\geq 3$, so the following expressions are well-defined. We distinguish three cases. \\

\textbf{Case 2.1:} $N_{\Theta}(C_{i+6k^2}^+)\cap \Sp(B_{i}') \neq \emptyset$ for some $i\in \{1,\ldots ,156k^3\}$.\\
Let $x\in N_{\Theta}(C_{i+6k^2}^+)\cap \Sp(B_{i}')$ and let $B_x$ be the block of $C_{i+6k^2}$ with $x\in N_{\Theta}(B_x^+)$. Let $\B = \{B_i',\ldots ,B_{i+6k^2-1}'\}\setminus \{B_x\}$, $\ell = \ell (B_{i+6k^2}')$, and $r= r(B_{i+6k^2}')$. Let $P_r$ be an $r-x$ path with interior vertices in $C_{i+6k^2}^+$. Note that $P_r$ is disjoint from $B$ for every $B\in \B$. Let $\B_3\subseteq \B$ be the set of blocks which contain three vertices of degree 2. It is easy to see that there exists a cycle $C$ in $G$ which contains $P_r$ and goes through every block in $\B$. If $|\B_3| \geq k$, then the union of $C$ and the blocks in $\B_3$ is a $k$-wiggly $\B_3$-necklace. So we may assume $|\B_3|<k$.\\
Let $\B' = \B\setminus \B_3$. Note that $x\in \ell P_1r$, so we can define $Q_1 = \ell P_1x$ and $Q_2 = xP_1r$. For every $B\in \B'$ there exists a vertex $x_B\in B$ such that $\pi (x_B) \in (Q_1\cup Q_2)\setminus \{\ell (B),r(B)\}$. We may assume that at least $\frac{1}{2}|\B'|$ of the vertices $\pi(x_B)$ with $B\in \B'$ are contained in $Q_1$. For every $B\in \B'$, let $z_B\in B$ such that $\pi (z_B) = r(B)$ and let $P_B$ be a $\pi(x_B)-\ell (B)$ path with interior vertices in $B^+$ and containing $z_B$. Let $Q'$ be the graph we get from $Q_1$ by adding the edges $\pi(x_B)\ell (B)$ for every $B\in \B'$ with $\pi (x_B)\in V(Q_1)$, see Figure~\ref{fig:thm_5_16_case_2_1}. The number of edges added is at least $\frac{1}{2}|\B'| \geq \frac{1}{2}(6k^2-k) \geq \frac{8}{3}k(k-1)+1$, so there exists a path $P'$ in $Q'$ containing at least $k$ of these edges by Lemma~\ref{lem:chord}. As in Case 1.1, we can modify $P'$ by replacing the edges $\pi(x_B)\ell (B)$ by the paths $P_B$ to obtain a path $P$ in $G$. Now $P$ and the cycle $Q_2\cup P_r$ are $k$-close.\\

\begin{figure}[]
        \centering
        \includegraphics[scale=.9]{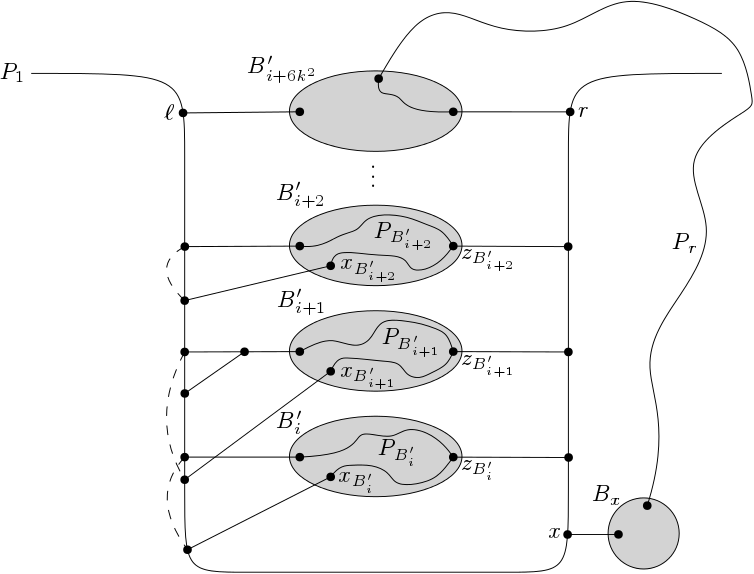}
    \caption{Proof of Theorem~\ref{thm:theta-isolated} Case 2.1}
    \label{fig:thm_5_16_case_2_1}
\end{figure}

\textbf{Case 2.2:} $N_{\Theta}(C_i^+)\setminus \Sp(B_{i+6k^2}') \neq \emptyset$ for some $i\in \{1,\ldots ,156k^3\}$.\\
Let $x\in N_{\Theta}(C_i^+)\setminus \Sp(B_{i+6k^2}')$ and let $B_x$ be the block of $C_{i}$ with $x\in N_{\Theta}(B_x^+)$. We may assume that $x$ is contained in the cycle $P_1\cup P_2$. Let $P_r$ be an $x-r(B_i')$-path with interior vertices in $C_i^+$. 
Let $\B = \{B_{i+1}',\ldots ,B_{i+6k^2}'\}\setminus \{B_x\}$ and $\B_3\subseteq \B$ the set of blocks with three vertices of degree 2. The vertices $x$ and $r(B_i')$ split the cycle $P_1\cup P_2$ into two paths $Q_1$ and~$Q_2$. We may assume that $Q_1$ contains $\ell (B)$ and $Q_2$ contains $r(B)$ for every $B\in \B$. 
Now we proceed as in the previous case.\\ 
If $|\B_3|\geq k$, then we can find a $k$-wiggly $\B_3$-necklace. So we may assume $|\B_3|<k$. Let $\B' = \B\setminus \B_3$. For every $B\in \B'$ there exists a vertex $x_B\in B$ such that $\pi (x_B) \in (Q_1\cup Q_2)\setminus \{\ell (B),r(B)\}$. We may assume that at least $\frac{1}{2}|\B'|\geq \frac{8}{3}k(k-1)+1$ of the vertices $\pi(x_B)$ with $B\in \B'$ are contained in $Q_1$. For every $B\in \B'$, let $z_B\in B$ such that $\pi (z_B) = r(B)$. As before, there exists a path $P$ which is disjoint from $Q_2\cup P_r$ and contains at least $k$ vertices $z_B$. Now $P$ and the cycle $Q_2\cup P_r$ are $k$-close.\\

\textbf{Case 2.3:} $N_{\Theta}(C_i^+)\subset \Sp(B_{i+6k^2}')$ and  $N_{\Theta}(C_{i+6k^2}^+)\cap \Sp(B_{i}') = \emptyset$ for all $i\in \{1,\ldots ,156k^3\}$.\\
Note that $N_{\Theta}(C_i^+)\subset \Sp(B_{i+6k^2}')$ implies that every component in $\C$ is $\Theta$-isolated. Let $C_i' = C_{12ik^2}$ for $i\in\{1,\ldots ,13k\}$, and $\C' = (C_1',\ldots ,C_{13k}')$. We show that $\C'$ is a $\Theta$-chain. For $i,j\in \{1,\ldots ,156k^3\}$ with $j\geq i+12k^2$, we have $N_{\Theta}(C_i^+) \subset \Sp (B_{i+6k^2}') \subseteq \Sp (B_{j-6k^2}')$ and $N_{\Theta}(C_j^+) \cap \Sp (B_{j-6k^2}')=\emptyset$, which implies $\Sp(C_i)\cap N_{\Theta}(C_j^+) = \emptyset$. Thus, $C_i$ and $C_j$ are not crossing, which implies that the components in $\C'$ are pairwise not crossing. Moreover, $\Sp (C_i) \subseteq \Sp (B_{i+6k^2}')\subset \Sp (B_j') \subseteq \Sp (C_j)$ for $i,j\in \{1,\ldots ,156k^3\}$ with $j\geq i+12k^2$. Thus, $C_i'\leq_{\Theta} C_{i+1}'$ for $i\in\{1,\ldots ,13k\}$ and $\C'$ is a $\Theta$-chain. By Theorem~\ref{thm:leaves}, we may assume that there are less than $8k$ components containing an endblock which is not 2-connected. Since~$\C'$ has length $13k$, it contains a subsequence of length $5k$ which is a special $\Theta$-chain. Now~$G$ is $k$-good by Lemma~\ref{lem:theta-chain}.
\end{proof}

\subsection{Few 2-connected endblocks in $G-\Theta^+$}

Combining the results of Section 5.2 and Section 5.3, the only remaining case in the proof of Theorem~\ref{thm:main} is when $G-\Theta^+$ does not contain many 2-connected endblocks. By Theorem~\ref{thm:leaves} we can assume that there are not many endblocks in $G-\Theta^+$. In this case we can find a path and a cycle which are $k$-close.

\begin{theorem}\label{thm:main-diam}
Let $G$ be a cubic 3-connected graph and $k$ a natural number. If the diameter of $G$ is at least $10^9k^{13}2^{9k^2}$, then $G$ is $k$-good.
\end{theorem}
\begin{proof}
Let $f(k) = 10^9k^{13}2^{9k^2}$ and let $u,v\in V(G)$ such that $u$ and $v$ have distance at least $f(k)$ in $G$. Let $\Theta =(P_1,P_2,P_3)$ be a shortest $u,v$-$\theta$-graph in $G$ and $H=G-\Theta^+$. We may assume $k\geq 3$ since $\Theta$ is 2-good.
By Lemma~\ref{lem:easy_specialcases}~(a) we can assume that there are less than $\frac{3}{2}k$ edges in $G-E(\Theta)$ with both ends in $V(\Theta)$. Thus, we have 
\begin{align*}
|E(\Theta , G-\Theta)| > 3(f(k)-1)-3k\,.
\end{align*} 
Let $v_0$ and $e_0$ denote the number of isolated vertices and isolated edges in $G-\Theta$, respectively. By Lemma~\ref{lem:easy_specialcases}~(c) and (d) we may assume $v_0<3k$ and $e_0<3k$. 
Note that $|E(\Theta , F(\Theta))| = 2|E(F(\Theta), H)|+3v_0+4e_0 < 2|E(F(\Theta), H)| + 21k$, and thus
\begin{align*}
|E(F(\Theta ), H)| > \frac{1}{2}(|E(\Theta , F(\Theta))| - 21k)\,.
\end{align*} 
Since $|E(\Theta , G-\Theta)| = |E(\Theta , H)| + |E(\Theta , F(\Theta))|$, this implies
\begin{align*}
|E(\Theta^+,H)| & = |E(\Theta , H)| + |E(F(\Theta), H)|\\
                & > |E(\Theta , H)| + \frac{1}{2}(|E(\Theta , F(\Theta))| - 21k)\\
                & \geq \frac{1}{2}(|E(\Theta , G-\Theta)|-21k)\\
                & > \frac{3}{2}(f(k)-1)-12k\,.
\end{align*}
Let $S=N_H(\Theta^+)$ and let $s_{1}$ denote the number of vertices of degree at most 1 in $H$. Each vertex in $S$ has degree at most 1 in $H$ or it is incident with exactly one edge in $E(\Theta^+,H)$, so $|S|\geq |E(\Theta^+,H)|-2s_{1}$. By Theorem~\ref{thm:leaves} we may assume $s_{1} < 8k$.
Thus, 
$$|S|\geq |E(\Theta^+,H)|-16k > \frac{3}{2}f(k)-29k > f(k)\,.$$
By Theorem~\ref{thm:theta-connecting} and Theorem~\ref{thm:theta-isolated}, we may assume that the number of 2-connected endblocks in $H$ is less than $21k^2+\frac{3}{2}k+5700k^6$. By Theorem~\ref{thm:leaves} we can assume that there are less than $8k$ endblocks in $H$ which are not 2-connected. In total, $H$ has less than $10^4k^6$ endblocks and in particular also less than $10^4k^6$ components. Let $K$ be a component of $H$ for which $|K\cap S|$ is maximal.
Since $|S| > f(k)$, we have 
\begin{align}\label{eq:ks}
|K\cap S| > \frac{|S|}{10^4k^6} > \frac{f(k)}{10^4k^6} = 2^{9k^2}10^5k^7 > 3 \cdot 2^{9k^2}10^4k^7\,.
\end{align} 
Let $\B$ be the set of blocks of $K$ which contain a vertex in $S$. Since we can 2-colour the blocks of $K$ such that any two blocks of the same colour are disjoint, there exists a subset of pairwise disjoint blocks $\B'\subseteq \B$ with $|\B'|\geq \frac{1}{2}|\B|$.\\
Suppose there exists a block $B$ in $H$ containing at least $2^{9k^2}$ vertices of $S$. By Theorem~\ref{thm:chord_adv}, there exists a path $P$ in $B$ containing at least $\frac{3}{2}k$ vertices in $S$. Now $P$ is $k$-close to one of the cycles $P_1\cup P_2$, $P_1\cup P_3$, or $P_2\cup P_3$. Thus we may assume that each block contains less than $2^{9k^2}$ vertices of $S$. Together with~(\ref{eq:ks}) this implies $|\B|\geq 3\cdot 10^4k^7$.\\
Let $\PP$ be a minimal collection of paths in $K$ such that at least one edge of each block in $\B'$ is contained in some path of $\PP$. It is easy to see that $\PP$ contains at most as many paths as $K$ has endblocks, thus $|\PP|<10^4k^6$. Thus, there exists a path $P\in \PP$ such that $P$ contains edges of at least 
$$\frac{|\B'|}{|\PP|} > \frac{\frac{1}{2}\cdot 3\cdot 10^4k^7}{10^4k^6} = \frac{3}{2}k$$ 
different blocks in $\B'$. We can modify $P$ so that it contains a vertex of $S$ in each of these blocks. Now $P$ is $k$-close to one of the cycles $P_1\cup P_2$, $P_1\cup P_3$, or $P_2\cup P_3$.
\end{proof}

The following corollary is an immediate application of Theorem~\ref{thm:main-diam} which concludes the proof of Theorem~\ref{thm:main}.

\begin{corollary}
Let $G$ be a cubic 3-connected graph and $m$, $k$ natural numbers with $k$ odd, and $f(k)= 2^{10^6k^{16}}$. If $|V(G)|\geq 3\cdot 2^{f(k)}$ then $G$ contains a cycle whose length is congruent to $m$ modulo $k$.
\end{corollary}
\begin{proof}
We may assume $k\geq 3$. The diameter of $G$ is at least $f(k)$. Note that $162^{13}<10^{29}$. By Bernoulli's inequality we have
\begin{align*}
    2^{\frac{1}{2}10^6k^{16}} =  \left(2^{\frac{1}{2}10^5k^{16}}\right)^{10} > \left(\frac{1}{2}10^5k^{16}\right)^{10} > 10^{46}k^{160} >  10^9(162k^8)^{13}\,.
\end{align*}
Since $\frac{1}{2}10^6 > 9\cdot (162)^2$, we have
\begin{align*}
    f(k) = 2^{\frac{1}{2}10^6k^{16}}\cdot 2^{\frac{1}{2}10^6k^{16}} > 10^9(162k^8)^{13}2^{9(162k^8)^2}\,.
\end{align*}
Thus, $G$ is $(162k^8)$-good by Theorem~\ref{thm:main-diam}. By Theorem~\ref{thm:kgood}, $G$ has a cycle whose length is congruent to $m$ modulo $k$. 
\end{proof}

\section{Counterexamples for 2-connected cubic graphs}

In some cases weaker conditions suffice to show that a graph contains a cycle whose length is congruent to $m$ modulo $k$. This is in particular true for small values of $k$. For example, Chen and Saito~\cite{chensaito} proved that every graph of minimum degree 3 contains a cycle whose length is divisible by 3. However, in this section we construct families of graphs which show that in general Theorem~\ref{thm:main} cannot be extended to 2-connected cubic graphs. The graphs we construct consist of two disjoint copies of a small graph that are joined by a so-called cross-ladder which is defined as follows.

\begin{definition}[Cross-ladder]
Let $N \geq 1$ be a natural number. A \textbf{cross-ladder} of length $3N$ is a graph consisting of a path $P=u_0u_1 \ldots u_{3N-1}u_{3N}v_{3N}v_{3N-1}\ldots v_1v_0$ and edges $u_{3i}v_{3i}$, $u_{3i+1}v_{3i+2}$, $u_{3i+2}v_{3i+1}$ for every $i\in \{0,\ldots ,N-1\}$.
\end{definition}

Note that every cross-ladder is a 2-connected subcubic graph. Each cross-ladder contains precisely four vertices of degree 2 ($u_0$, $v_0$, $u_{3N}$, and $v_{3N}$) and they induce a matching in the cross-ladder. We now define an operation that allows us to connect two disjoint graphs using a cross-ladder.

\begin{definition}[$G_1 \otimes_{N} G_2$]
Let $N\geq 1$ be a natural number, $G_1$ and $G_2$ 2-connected graphs with $x_i,y_i\in V(G_i)$ such that $d(x_i)=d(y_i)=2$ and $d(v) = 3$ for $i\in \{1,2\}$ and $v\in V(G_i)\setminus\{x_i,y_i\}$. Let $L$ be a cross-ladder of length $3N$ and $u_0,v_0,u_{3N},v_{3N}\in V(L)$ distinct vertices of degree 2 such that $u_0v_0$, $u_{3N}v_{3N}\in E(L)$. We define $G_1 \otimes_{N} G_2$ as the graph we obtain from the disjoint union of $G_1$, $G_2$, and $L$ by adding the edges $x_1u_0$, $y_1v_0$, $x_2u_{3N}$, and $y_2v_{3N}$.
\end{definition}

\begin{figure}[b]
    \centering
    \includegraphics[width=\textwidth]{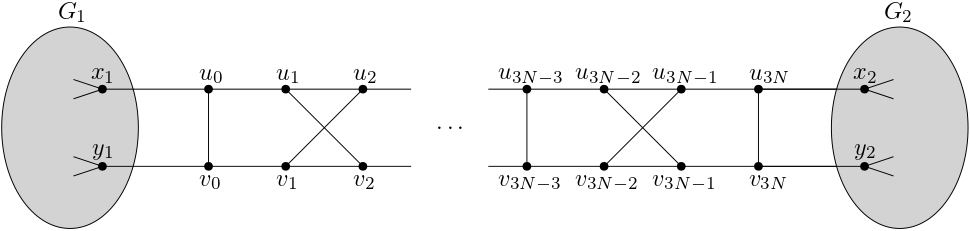}
    \caption{$G_1 \otimes_{N} G_2$}
    \label{fig:crossladder_G_1_G_2}
\end{figure}

It is easy to see that $G_1 \otimes_{N} G_2$ is a cubic 2-connected graph, see Figure~\ref{fig:crossladder_G_1_G_2}. Moreover, if $G_1+x_1y_1$ and $G_2+x_2y_2$ are planar, then also $G_1 \otimes_{N} G_2$ is planar. The following lemma shows that under some mild assumption on $G$ we can control the cycle lengths which are divisible by 3 in $G\otimes_{N} G$.

\begin{lemma}\label{lem:counterex-cyclelengths}
Let $N\geq 1$ be a natural number and $G$ a 2-connected graph with $x,y\in V(G)$ such that $d(x)=d(y)=2$ and $d(v) = 3$ for $v\in V(G)\setminus\{x,y\}$. Suppose that $G$ contains no $x-y$ path whose length is divisible by 3. Let $G' = G\otimes_N G$ and let $G_1$ and $G_2$ denote the two disjoint copies of $G$ in $G'$. If $C$ is a cycle in $G'$ whose length is divisible by 3, then either $C$ is a cycle in $G_1$ or $G_2$, or $C$ has length $6N+4+p_1+p_2$ where $p_1$ and $p_2$ are lengths of $x-y$ paths in $G$.
\end{lemma}

\begin{proof}
Let $x_1,y_1\in V(G_1)$ and $x_2,y_2\in V(G_2)$ denote the vertices corresponding to $x$ and $y$ in $G$.
Let $H$ be the subgraph of $G'$ which is induced by the copy of the cross-ladder of length $3N$ and the vertices $x_1,y_1,x_2$, and $y_2$. Note that $H$ contains no cycles of length divisible by 3, so we can assume $C$ contains a vertex in at least one of $G_1$ and $G_2$, say $G_1$. 
We may assume $C$ is not a cycle in $G_1$, so it contains both $x_1$ and $y_1$. Note that the length of every $x_1-y_1$ path in $H$ is congruent to 0 modulo 3. 
If $C$ is contained in $G_1\cup H$, then 
$$|E(C)| = |E(C) \cap E(G_1)|+|E(C) \cap E(H)| \equiv |E(C)\cap E(G_1)| \not\equiv 0 \mod{3}\,,$$
since $|E(C)\cap E(G_1)|$ is the length of an $x-y$ path in $G$. Thus we can assume that $C$ also contains vertices in $G_2$. 
Now it is easy to see that $|E(C) \cap E(H)|=6N+4$. Let $p_1 = |E(C)\cap E(G_1)|$ and $p_2 =|E(C)\cap E(G_2)|$. It follows that $|E(C)|= 6N+4+p_1+p_2$. 
\end{proof}

Note that in Lemma~\ref{lem:counterex-cyclelengths}, both $p_1$ and $p_2$ are lengths of $x-y$ paths in $G$ and thus not divisible by 3. Since $|E(C)|$ is divisible by 3, we have $p_1 + p_2 \equiv 2$ modulo 3 and both $p_1$ and $p_2$ are congruent to 1 modulo 3.

\begin{figure}[h]
    \centering
      \begin{subfigure}[b]{0.24\textwidth}
        \includegraphics[width=\textwidth]{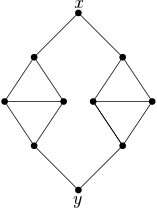}
        \caption{$H_1$}
        \label{fig:crossladder_end_1}
    \end{subfigure}
    \hspace{2cm}
    \begin{subfigure}[b]{0.16\textwidth}
        \includegraphics[width=\textwidth]{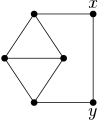}
        \caption{$H_2$}
        \label{fig:crossladder_end_2}
    \end{subfigure}
    \hspace{2cm}
     \begin{subfigure}[b]{0.31\textwidth}
        \includegraphics[width=\textwidth]{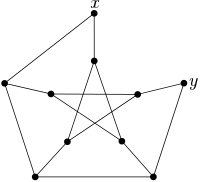}
        \caption{$H_3$}
        \label{fig:crossladder_end_petersen}
    \end{subfigure}
    \caption{Three different building blocks used in the proof of Theorem~\ref{thm:2-con_counter}}
    \label{fig:crossladder_three_graphs}
\end{figure}

\begin{theorem} \label{thm:2-con_counter}
Let $m$, $k$, and $N$ be natural numbers with $k\geq 12$. If $m$ and $k$ are divisible by 3, then there exists a 2-connected cubic graph $G$ with at least $N$ vertices such that no cycle of $G$ has length congruent to $m$ modulo $k$.
\end{theorem}
\begin{proof}
We may assume $m\in \{0,\ldots ,k-1\}$. We define a number $N'$ and a graph $G$ depending on $k$, $m$, and $N$ as follows.
\begin{itemize}
    \item If $m \notin \{3,9\}$, let $N' \geq N$ be a natural number such that $6N'+12 \not\equiv m \mod{k}$ and let $G=H_1$ be the graph in Figure~\ref{fig:crossladder_three_graphs}(a).
    \item If $m =9$, let $N' \geq N$ be a natural number such that $N' \equiv 1$ (mod $k$) and let $G=H_2$ be the graph in Figure~\ref{fig:crossladder_three_graphs}(b).
    \item If $m=3$, let $N' \geq N$ be a natural number such that $N' \equiv -1$ (mod $k$) and let $G=H_3$ be the Petersen Graph with one edge removed, see Figure~\ref{fig:crossladder_three_graphs}(c).
\end{itemize}
Let $x$ and $y$ denote the two vertices of degree 2 in $G$. Note that in each case there is no $x-y$ path in $G$ whose length is divisible by 3.
Let $G' = G\otimes_{N'}G$ and let $G_1$, $G_2$ denote the two copies of $G$ in $G'$. Let $C$ be a cycle in $G'$ whose length is divisible by 3. We distinguish between the three cases above. \\

\textbf{Case 1:} $m \notin \{3,9\}$.\\
The only cycles of length divisible by 3 in $H_1$ have length 3 or 9. Thus, if $C$ is contained in $G_1$ or $G_2$, then $C$ has length 3 or 9. By Lemma~\ref{lem:counterex-cyclelengths}, we may assume that $|E(C)|=6N'+4+p_1+p_2$ where $p_1,p_2\in \{4,5\}$. Since $|E(C)|$ is divisible by 3, it follows that $|E(C)|=6N'+12$, so the length of $C$ is not congruent to $m$ modulo $k$ by the choice of $N'$.  \\

\textbf{Case 2:} $m=9$.\\
If $C$ is contained in $G_1$ or $G_2$, then $C$ has length 3 or 6. By Lemma~\ref{lem:counterex-cyclelengths}, we may assume that $|E(C)|=6N'+4+p_1+p_2$ where $p_1,p_2\in \{1,4,5\}$. Since $|E(C)|$ is divisible by 3, we have $|E(C)|\in \{6N'+6, 6N'+9, 6N'+12\}$. Thus $|E(C)|$ is congruent to 12, 15, or 18 modulo $k$. Since $k\geq 12$, the length of $C$ is not congruent to 9 modulo $k$.\\

\textbf{Case 3:} $m=3$.\\
Since $H_3$ has no cycle whose length is congruent to 3 modulo $k$, we can assume by Lemma~\ref{lem:counterex-cyclelengths} that $|E(C)|=6N'+4+p_1+p_2$ where $p_1,p_2\in \{4,5,7,8\}$. Thus $|E(C)|\in\{6N'+12, 6N'+15, 6N'+18\}$ and $|E(C)|$ is congruent to 6, 9, or 12 modulo $k$. Since $k\geq 12$, the length of $C$ is not congruent to 3 modulo $k$.
\end{proof}

We conclude this paper with the following open problem.

\begin{question}
For which natural numbers $m$ and $k$ does every sufficiently large 2-connected cubic graph contain a cycle whose length is congruent to $m$ modulo $k$?
\end{question}

\bibliographystyle{amsplain}

\begin{thebibliography}{99}

\bibitem{bollobas} \textsc{B. Bollobás}, \emph{Cycles modulo k}, Bull. London Math. Soc. 9 (1977), 97-98. 

\bibitem{bondy} \textsc{J.A. Bondy, A. Vince}, \emph{Cycles in a Graph Whose
Lengths Differ by One or Two}, J. Graph Theory 27 (1998), 11-15.

\bibitem{caishreve} \textsc{X. Cai, W. Shreve}, \emph{Pancyclicity mod $k$ of claw-free graphs and $K_{1,4}$-free graphs}, Discrete Mathematics 230 (2001), 113-118.

\bibitem{chensaito} \textsc{G.T. Chen, A. Saito}, \emph{Graphs with a Cycle of Length Divisible by Three}, J. Combin. Theory Ser. B 60 (1994), 277-292.

\bibitem{chung} \textsc{F. Chung},  \emph{Open  problems  of  Paul  Erd\H{o}s  in  graph  theory},  J.  Graph  Theory 25 (1997), 3-36.

\bibitem{deanlesniak} \textsc{N. Dean, L. Lesniak, A. Saito}, \emph{Cycles of length 0 modulo 4 in graphs}, Discrete Mathematics 121 (1993), 37-49.

\bibitem{diwan} \textsc{A.A. Diwan}, \emph{Cycles of even lengths modulo $k$}, J. Graph Theory 65 (2010), 246-252.

\bibitem{erdos} \textsc{P. Erd\H{o}s}, \emph{Some recent problems and results in graph theory, combinatorics, and number theory}, Proceedings of Seventh S-E Conference on Combinatorics, Graph Theory and Computing, Utilitas   Mathematica, Winnipeg (1976), 3–14.

\bibitem{erdossz} \textsc{P. Erd\H{o}s, G. Szekeres}, \emph{A combinatorial problem in geometry}, Compositio Mathematica 2 (1935), 463-470.

\bibitem{fan} \textsc{G. Fan}, \emph{Distribution of Cycle Lengths in Graphs}, 
J. Combin. Theory Ser. B 82 (2002), 187-202.

\bibitem{langwalther} \textsc{R. Lang, H. Walther}, \emph{\"Uber L\"angste Kreise in regul\"aren Graphen}, Beitr\"age zur Graphentheorie, Kolloquium, Manebach 1967, Teubner, Leipzig (1968), 91-98. 

\bibitem{liuma} \textsc{C.H. Liu, J. Ma}, \emph{Cycle lengths and minimum degree of graphs}, J. Combin. Theory, Ser. B 128 (2018), 66-95.

\bibitem{ma} \textsc{J. Ma}, \emph{Cycles with consecutive odd lengths}, European J. Combin 52 (2016), 74-78.

\bibitem{mei} \textsc{L. Mei, Y. Zhengguang}, \emph{Cycles of length 1 modulo 3 in graph}, Discrete Applied Mathematics 113 (2001), 329-336.

\bibitem{saito} \textsc{A. Saito}, \emph{Cycles of length 2 modulo 3 in graphs}, Discrete Mathematics 101 (1992), 285-289.

\bibitem{sudakov2} \textsc{B. Sudakov,  J. Verstraëte}, \emph{Cycle lengths in sparse graphs}, Combinatorica 28 (2008), 357-372.

\bibitem{sudakov} \textsc{B. Sudakov,  J. Verstraëte}, \emph{The extremal function for cycles of length $\ell$ mod~$k$}, Electronic J. Combin. 24 (2017), $\#$P1.7.

\bibitem{girth} \textsc{C. Thomassen}, \emph{Girth in graphs}, J. Combin. Theory Ser. B 35 (1983), 129-141.

\bibitem{thomassenpath} \textsc{C. Thomassen}, \emph{Graph Decomposition with Applications to Subdivisions and Path Systems Modulo $k$}, J. Graph Theory 7 (1983), 261-271.

\bibitem{tutte} \textsc{W.T. Tutte}, \emph{How to draw a graph}, Proc. London Math. Soc. 13 (1989), 175-188.

\bibitem{verstarit} \textsc{J. Verstraëte}, \emph{On Arithmetic Progressions of Cycle Lengths in Graphs}, Combinatorics, Probability and Computing 9 (2000), 369-373.

\bibitem{verte} \textsc{J. Verstraëte}, \emph{Unavoidable cycle lengths in graphs}, J. Graph Theory 49 (2005), 151-167.


\end{thebibliography}


\begin{aicauthors}
\begin{authorinfo}[kasper]
  Kasper Szabo Lyngsie\\
  Department of Applied Mathematics and Computer Science\\
  Technical University of Denmark\\
  Lyngby, Denmark\\
  kasperszabo\imageat{}hotmail\imagedot{}com
\end{authorinfo}
\begin{authorinfo}[martin]
  Martin Merker\\
  Department of Applied Mathematics and Computer Science\\
  Technical University of Denmark\\
  Lyngby, Denmark\\
  merkermartin\imageat{}gmail\imagedot{}com
\end{authorinfo}
\end{aicauthors}

\end{document}